\documentclass[11pt,a4paper]{imsart}

\usepackage{amsmath,amsthm,amssymb,color}
\usepackage[authoryear]{natbib}
\usepackage{booktabs}

\usepackage{bbm}
\usepackage{mathrsfs}
\usepackage[breaklinks=true]{hyperref}

\numberwithin{equation}{section}
\allowdisplaybreaks[4]

\theoremstyle{plain}
\newtheorem{theorem}{Theorem}[section]

\newtheorem{lemma}{Lemma}[section]
\newtheorem{corollary}{Corollary}[section]

\theoremstyle{definition}

\newtheorem{example}{Example}[section]
\newtheorem{remark}{Remark}[section]

\makeatletter

\newcount\minute
\newcount\hour
\newcount\hourMins
\def\now{%
\minute=\time%
\hour=\time \divide \hour by 60%
\hourMins=\hour \multiply\hourMins by 60%
\advance\minute by -\hourMins%
\zeroPadTwo{\the\hour}:\zeroPadTwo{\the\minute}%
}
\def\zeroPadTwo#1{\ifnum #1<10 0\fi#1}

\renewcommand{\cite}{\citet}

\def\^#1{\ifmmode {\mathaccent"705E #1} \else {\accent94 #1} \fi}
\def\~#1{\ifmmode {\mathaccent"707E #1} \else {\accent"7E #1} \fi}

\def\*#1{#1^\ast}
\edef\-#1{\noexpand\ifmmode {\noexpand\bar{#1}} \noexpand\else \-#1\noexpand\fi}
\def\>#1{\vec{#1}}
\def\.#1{\dot{#1}}

\def\atop{\@@atop}
\def\*#1{\mathscr{#1}}

\renewcommand{\leq}{\leqslant}
\renewcommand{\geq}{\geqslant}

\newcommand{\eq}{\eqref}

\newcommand{\IE}{\mathbbm{E}}
\newcommand{\IP}{\mathbbm{P}}

\def\be#1{\begin{equation*}#1\end{equation*}}
\def\ben#1{\begin{equation}#1\end{equation}}
\def\bes#1{\begin{equation*}\begin{split}#1\end{split}\end{equation*}}
\def\besn#1{\begin{equation}\begin{split}#1\end{split}\end{equation}}

\def\ceil#1{{\lceil#1\rceil}}

\def\beqn#1\eeqn{\begin{align}#1\end{align}}
\def\beq#1\eeq{\begin{align*}#1\end{align*}}

\usepackage{graphicx}
\usepackage{latexsym}
\usepackage{amsmath,amsthm,amssymb,amscd}
\usepackage{epsf,amsmath}

\def\E{{\IE}}

\def\P{{\IP}}

\def\min{{\rm min}}
\def\I{{\rm I}}

\def\blfootnote{\xdef\@thefnmark{}\@footnotetext}

\makeatother

\begin{document}

\begin{frontmatter}
\title{Poisson Approximation for Two Scan Statistics with Rates of Convergence}

\runtitle{Poisson Approximation for Two Scan Statistics}

\begin{aug}
\author{\fnms{Xiao}
\snm{Fang}\thanksref[1]{t1}\ead[label=e1]{stafx@nus.edu.sg}}
and
\author{\fnms{David} \snm{Siegmund}\thanksref[2]{t2}\ead[label=e2]{siegmund@stanford.edu}}
\thankstext[1]{t1}{Partially supported by the NUS-Overseas Postdoctoral Fellowship from the National University of Singapore.}
\thankstext[2]{t2}{Partially supported by the National Science Foundation.}


 \runauthor{X. Fang  and D. Siegmund}

\affiliation{National University of Singapore and Stanford University}

\address{Department of Statistics and Applied Probability\\
National University of Singapore\\
6 Science Drive 2\\
Singapore 117546 \\
Republic of Singapore\\
\printead{e1}\\
\phantom{E-mail:\ }}

\address{Department of Statistics\\
Sequoia Hall\\
390 Serra Mall\\
Stanford University\\
Stanford, CA 94305-4065\\
\printead{e2}\\
\phantom{E-mail:\ }}

\end{aug}

\begin{abstract} 
As an application of Stein's method for 
Poisson approximation, we prove rates of convergence
for the tail probabilities of two 
scan statistics that have been suggested  
for detecting local signals in 
sequences of independent random variables subject to
possible change-points.  Our formulation
deals simultaneously with ordinary and with large deviations.

\end{abstract}

\begin{keyword}[class=AMS]
\kwd[Primary ]{60F05,62E20} \kwd[; secondary ]{62L10}
\end{keyword}

\begin{keyword}
\kwd{Stein's method}
\kwd{Poisson approximation}
\kwd{total variation distance}
\kwd{relative error}
\kwd{rate of convergence}
\kwd{scan statistics}
\kwd{change-point analysis}
\kwd{exponential family}
\end{keyword}

\end{frontmatter}

\section{Introduction}

Let $\{X_1,\dots, X_n\}$ be a sequence of random variables. A widely 
studied problem is to test the hypothesis that the $X$'s are independent 
and identically distributed against the alternative that for some 
$0\leq i<j\leq n$, $\{X_{i+1},\dots, X_j\}$ have a distribution that differs 
from the distribution of the other $X$'s. If $t:=j-i$ is assumed known and the
change in distribution is a shift in the mean, one common suggestion 
to detect the change is
the statistic
\ben{\label{0.1}
M_{n; t}=\max_{1\leq i\leq n-t+1} (X_i+\dots +X_{i+t-1}).
}
See \cite{GlNaWa01} for an introduction to scan statistics.

When $t$ is unknown but the distributions of the $X$'s are otherwise
completely specified, the maximum log likelihood ratio statistic is
\ben{\label{0.2}
\max_{0\leq i<j \leq n}(S_j-S_i)
}
where
\ben{\label{0.3}
S_i=\sum_{k=1}^i \log[f_1(X_k)/f_0(X_k)]
}
and $f_0$ ($f_1$ resp.) is the density function of $X$ under the null hypothesis (alternative hypothesis resp.).
Appropriate  statistics when the
distributions involve unknown parameters can be found,
for example, in \cite{Ya93}.

Asymptotic $p$ values of test statistics \eq{0.1} and \eq{0.2} have  
been derived as $n\rightarrow \infty$ under certain distributional 
assumptions on $X_1$.
See, for example, \cite{ArGoWa90}, \cite{Ha07} and \cite{Si88}.
The statistic \eq{0.2} has also been studied for its role
in queueing theory, where it has the interpretation of the
maximum waiting time among the first $n$ customers of a
single server queue (cf. \cite{Ig72}). 
However, except for \eq{0.1} in the special case when
$X_1$ is a Bernoulli variable (cf. \cite{ArGoWa90} and \cite{Ha07}),
and for \eq{0.2} when the problem is scaled so that the probability is approximately zero (cf. \cite{Si88}),
the rate of convergence for 
these approximations is unknown. In this paper, we establish rate of 
convergence of tail approximations for both statistics \eq{0.1} and \eq{0.2} 
under the assumption that $X_1$ comes from an exponential 
family of distributions.  
The error in our approximation is relative error, hence
is applicable when the probability is small as well as when it converges to a 
positive limit.

In practice simulations have been widely used to justifiy the accuracy
of the approximations suggested here.  
The sample size $n$ used in those simulations is typically a few thousands at most, partly because 
the simulation would take too long for larger $n$.
We have not seen any related work trying to infer a convergence rate by simulation results.
The constants arising from
our calculations are undoubtedly much too large to be an
alternative source to  justify use of the
approximations in practice.  We view the value of our approximations as
providing
understanding of the relations of various parameters involved in the
approximations, and in particular the uniformity of the validity
of the approximation for both large and ordinary deviations. 

In the next section, we state our main results. Section 3 contains an introduction to our main technique, Stein's method, and the proof of our main results. We discuss related problems in Section 4.

\section{Main results}

\subsection{Scan statistics with fixed window size}

Let $\{X_1,\dots, X_n\}$ be independent,
identically distributed  random variables with distribution
function $F$ and $\E X_1=\mu_0$. 
Suppose the distribution of $X_1$ can
be imbedded in  an exponential family of probability measures $\{F_\theta: \theta\in \Theta\}$ where 
$\Theta$ is an open interval in $\mathbb{R}$ containing $0$, and
\ben{\label{1.2}
d F_\theta(x)=e^{\theta x-\Psi(\theta)} dF(x).
}
It is known that the mean and variance of $F_{\theta}$ are $\Psi'(\theta)$ 
and $\Psi''(\theta)$ respectively. We assume $F_\theta(x)$ is non-degenerate, i.e., $\Psi''(\theta)>0$ for all $\theta\in \Theta$. From $X_1\sim F$ and $\E X_1=\mu_0$, we have $\Psi(0)=0$ and $\Psi'(0)=\mu_0$.

Let $a>\mu_0$ be given. Assume that there exists $\theta_a\in \Theta$ 
such that $\Psi'(\theta_a) = a$.
For a positive integer $t< n$,
and for $M_{n; t}$ defined in \eq{0.1},
we are interested in calculating approximately 
the probability $\P(M_{n;t}\geq at)$.   
In the following theorem, we provide a Poisson approximation result with 
rate of convergence. 
We consider the following two cases: 
\begin{itemize}
\item[Case 1:] There exists $\theta_{a'}\in \Theta$ such that $\theta_{a'}>\theta_a$ (thus $\Psi'(\theta_{a'})=a'>a$) and
\ben{\label{702}
\sup_{\theta_a\leq \theta \leq \theta_{a'}} \int_{-\infty}^\infty |\varphi_\theta (x)|^{\nu}dx<\infty\ \text{for some positive integer}\  \nu,
}
where $\varphi_\theta$ is the characteristic function of $F_{\theta}$.

\item[Case 2:] $X_1$ is an integer-valued random variable with span $1$, where the span is defined to be the largest value of $\Delta$ such that
\ben{\label{701}
\sum_{k\in \mathbb{Z}}\P(X_i= k\Delta +w)=1 \ \text{for some}\  w\in \mathbb{Z}.
 }
\end{itemize}
We remark that \eq{702} is a smoothness condition on $F$ (cf. Condition 1.4 of \cite{DiFr88}). Note also that any lattice random variable, i.e., that satisfying \eq{701} with $w\in \mathbb{R}$ instead of $w\in \mathbb{Z}$, can be reduced to Case 2 by linear transformation.

For the statement in the following theorem, we define $\sigma_a^2=\Psi''(\theta_a)$ and $\lceil at \rceil=\inf \{v\in \mathbb{Z}: v\geq at \}$. Let $\{X_i, X_i^a: i\geq 1\}$ be independent with $X_i\sim F$ and $X_i^a\sim F_{\theta_a}$, and let $D_k=\sum_{i=1}^k (X_i^a-X_i)$ for $k\geq 1$.


\begin{theorem}\label{t1}
Under the assumptions given above,
for some constant $C$ depending only on the exponential family \eq{1.2}, 
$\mu_0$, and $a$, we have
\ben{\label{t1-1}
\big| \P(M_{n; t}\geq at) - (1-e^{-\lambda}) \big|\leq C(\frac{(\log t)^2}{t}+\frac{(\log t\wedge \log(n-t))}{n-t}) (\lambda \wedge 1),
}
where for Case 1,
\ben{\label{t1-2}
\lambda=\frac{ (n-t+1) e^{-[a\theta_a-\Psi(\theta_a)]t}}{\theta_a \sigma_a (2\pi t)^{1/2}}\exp[-\sum_{k=1}^\infty \frac{1}{k} \E(e^{-\theta_a D_k^+})],
}
and for Case 2,
\besn{\label{t1-3}
&\lambda=\frac{ (n-t+1) e^{-(a\theta_a-\Psi(\theta_a))t} e^{-\theta_a (\lceil at\rceil -at)}}
{(1-e^{-\theta_a})\sigma_a (2\pi t)^{1/2}}\exp[-\sum_{k=1}^\infty \frac{1}{k} \E(e^{-\theta_a D_k^+})].
}
\end{theorem}

\begin{remark}\label{r2}
The various expressions entering into $\lambda$ will be explained 
below.  Here it is important to note that provided $n-t$ and 
$t$ are large the error
of approximation is relative error, valid when $n$ is
relatively small, so $\lambda$ is near zero, 
and when $\lambda$ is bounded away from zero. 
Although it is possible to 
trace through the proof of Theorem \ref{t1} and obtain a numerical value for the constant $C$ in \eq{t1-1},
it would be too large for practical purposes. Therefore, we do not pursue it here.
\end{remark}

\begin{remark}
\cite{ArGoWa90} obtained a bound for $|\P(M_{n; t}\geq at) - (1-e^{-\lambda})|$ for
independent, identically distributed  Bernoulli random 
variables. They do not restrict the threshold ($at$ in our case) to grow linearly in $t$ with fixed slope.
For fixed $a$, their bound is of the form (cf. equations (11)--(13) of \cite{ArGoWa90})
\be{
C(e^{-ct}+\frac{t}{n})(\lambda\wedge 1).
}
Compared to their result, Theorem \ref{t1} applies to 
more general distributions and recovers typical limit theorems in the literature on 
scan statistics.  As $t, n-t\rightarrow \infty$, Theorem \ref{t1}  
guarantees the 
relative error in \eq{t1-1} goes to $0$. See, for example, Theorem 1 
of \cite{ChZh07}. 
\end{remark}

\begin{remark}\label{r3}  The infinite series appearing in the
definition of $\lambda$ is derived as an application of classical
random walk results of Spitzer.  
It arises probabilistically in the proof of Theorem 2.1 in the
form $\E [1-\exp\{-\theta_a D_{\tau_+}\}]/\E(\tau_+)$,
where $\tau_+ = \inf\{ t : D_t > 0\}$.
The series form is useful for numerical computation.
Let $g(x)=\E e^{i x D_1}$ and $\xi(x)=\log\{1/[1-g(x)]\}$.
\cite{Wo79} proved that for Case 1 of Theorem \ref{t1},
\besn{\label{ssstar1}
\sum_{k=1}^\infty \frac{1}{k} \E(e^{-\theta_a D_k^+})=&-\log[(a-\mu_0)\theta_a]-\frac{1}{\pi}\int_0^\infty \frac{\theta_a^2[I\xi(x)-\frac{\pi}{2}]}{x(\theta_a^2+x^2)} dx \\
&+ \frac{1}{\pi} \int_0^\infty \frac{\theta_a\{R\xi(x)+\log[(a-\mu_0)x]\}}{\theta_a^2+x^2}dx
}
where $R$ and $I$ denote real and imaginary parts.
\cite{TuSi99} proved that for Case 2 of Theorem \ref{t1},
\besn{\label{ssstar2}
&\sum_{k=1}^\infty \frac{1}{k} \E(e^{-\theta_a D_k^+})\\
=& -\log(a-\mu_0)
+\frac{1}{2\pi} \int_0^{2\pi} \left\{ \frac{\xi(x) e^{-\theta_a -ix}}{1-e^{-\theta_a-ix}} 
+\frac{\xi(x) +\log[(a-\mu_0)(1-e^{ix})]}{1-e^{ix}} \right\} dx.
}
The right-hand sides of \eq{ssstar1} and \eq{ssstar2} can be calculated by numerical integration.
For example, \cite{Wo79} calculated the right-hand side of \eq{ssstar1} for normal, gamma and chi-squared distributions, and \cite{TuSi99} calculated the right-hand side of \eq{ssstar2} for binomial distributions.
\end{remark}

\begin{remark}\label{r4}
For $M_{n;t}$ defined in \eq{0.1} and $b>0$, \cite{DeKa92} proposed 
the simple approximation  to $\P(M_{n;t}\geq b)$ given by $1-e^{-\lambda}$, 
where (cf. Theorem 2 of \cite{DeKa92})
\be{
\lambda=(n-t+1)\P(X_1+\dots+ X_t\geq b).
}
Similar approximations have also been considered for more complicated biological models. See \cite{ChKa00}, \cite{KaCh00} and \cite{ChKa07}. Such a simple approximation requires
specific conditions on the relation of $b$,  $t$ and $n$ and does not hold when
$b$ is proportional to $t$, which leads to the `clumping' phenomenon.
See, for example, Section 4.2 of \cite{ArGoGo90} or the book by \cite{Al89}.  In 
applications one must judge whether the
appropriate scaling relations hold for specific values
of $t$ and $b$.  In this regard it is interesting to note that our
approximation becomes the Dembo-Karlin approximation when the scaling relations of
\cite{DeKa92} are satisfied.
\end{remark}

Next, we compute the limiting probability $1-e^{-\lambda}$ in \eq{t1-1} explicitly for normal and Bernoulli random variables. We show that the limiting probability is close to the true probability by using simulation and known results.

\begin{example}
Suppose
$X_1\sim N(0,1)$. We have that $X_1^a\sim N(a,1)$, $D_1\sim N(a, 2)$ and in 
the definition of 
$\lambda$ in \eq{t1-2},
\bes{
\sum_{k=1}^\infty \frac{1}{k} \E(e^{-\theta_a D_k^+})&=2\sum_{k=1}^\infty \frac{1}{k} \Phi(-a\sqrt{k/2})\\
&=-\log[(a-\mu_0)^2\nu(\sqrt{2}a)]
}
where $\Phi$ is the standard normal distribution function and the function $\nu(x)$ was defined in (4.38) of \cite{Si85}.
It was shown there that
$\nu(x)= \exp(-c x) + o(x^2)$ as $x\rightarrow 0$
for $c \approx 0.583$, while $x^2 \nu(x) /2 \rightarrow 1$ as $x \rightarrow \infty$. 
\cite{SiYa07} indicate that a very simple and  good approximation is
\be{
\nu(x)\approx [(2/x)\{\Phi(x/2)-1/2\}]/\{(x/2)\Phi(x/2)+\phi(x/2)\}
} 
where $\phi$ is the standard normal density function.
Table \ref{table1} presents a numerical study with different values of $n, t$ and $a$. The limiting probability $1-e^{-\lambda}$ is denoted by $p_1$. The values of $p_2$ are simulated with 10,000 repetitions each. We can see from the table that $p_1$ is very close to the true probability.
\begin{table}
\begin{center}
\begin{tabular}{| l | l | l | l | l | l | l}
\hline
$n$ & $t$  & $a$ & $p_1$ & $p_2$ \\ \hline
1000 & 50  & 0.2 & 0.9315 & 0.9594\\ \hline
1000 & 50  & 0.4 & 0.2429 & 0.2624\\ \hline
1000 & 50  & 0.5 & 0.0331 & 0.0334 \\ \hline
2000 & 50  & 0.5 & 0.0668  & 0.0672 \\ \hline
\end{tabular}
\end{center}
\caption{Example 2.1}
\label{table1}
\end{table}
\end{example}

\begin{example}\label{example2}
Let $\{X_1,\dots, X_n\}$ be a sequence of independent Bernoulli random variables with $\P(X_i=1)=\mu_0$ for all $i$ where $0<\mu_0<1$. The distribution of $X_1$ can be imbedded in an exponential family of probability measures $\{F_\theta: \theta\in \mathbb{R}\}$ where $F_\theta$ is defined as in \eq{1.2} with
\ben{\label{e1-2}
\Psi(\theta)=\log(1+\frac{\mu_0 e^\theta}{1-\mu_0})+\log(1-\mu_0).
}
For $0<p<1$, define 
\ben{\label{e1-1}
\theta_p=\log(\frac{p}{1-p})-\log(\frac{\mu_0}{1-\mu_0}).
}
It is straightforward to check that
\be{
\Psi'(\theta_p)=p,\ \theta_{\mu_0}=0,\ \Psi(\theta_{\mu_0})=0.
}
Let $1>a>\mu_0$ (see Corollary \ref{c1} for the relatively easier case where $a=1$).
From \eq{t1-3} and \eq{ssstar2}, we have
\besn{\label{e1-3}
\lambda=&\frac{ (n-t+1) e^{-(a\theta_a-\Psi(\theta_a))t} e^{-\theta_a (\lceil at\rceil -at)}}
{(1-e^{-\theta_a})[2a(1-a)\pi t]^{1/2}}\times (a-\mu_0)\\
&\times 
\exp\left(-\frac{1}{2\pi} \int_0^{2\pi} \left\{ \frac{\xi(x) e^{-\theta_a -ix}}{1-e^{-\theta_a-ix}} 
+\frac{\xi(x) +\log[(a-\mu_0)(1-e^{ix})]}{1-e^{ix}} \right\} dx\right),
}
where $\theta_a$ and $\Psi(\theta_a)$ are defined in \eq{e1-1} and \eq{e1-2},
\be{
g(x)=a(1-\mu_0) e^{i x} +(1-a)\mu_0 e^{-i x} +a\mu_0 +(1-a)(1-\mu_0)
}
and $\xi(x)=\log\{1/[1-g(x)]\}$.

Let $t<n$, and let $M_{n;t}$ be defined as in \eq{0.1}.
The bound \eq{t1-1} suggests the following approximation to $\P(M_{n;t}\geq at)$:
\besn{\label{e1-4}
\P(M_{n;t}\geq at)\approx 1-e^{-\lambda},
}
where $\lambda$ is defined in \eq{e1-3}.
Table \ref{table2} presents a numerical study with different values of $n, t$ and $a$. The probability $p_1$ is calculated by the right-hand side of \eq{e1-4}. The values of $p_2$ are found in Table 1 of \cite{Ha07} and are shown there to be very accurate. We can see from the table that $p_1$ is very close to the true probability.
The derivation in \cite{Ha07} uses the distribution function of 
\be{
Z_k:=\max\{T_1,\dots, T_{kt+1}\}\ \text{for}\  k=1\ \text{and}\ 2,
}
where $T_\alpha=X_\alpha+\dots X_{\alpha+t-1}$. However, the distribution functions of $Z_k$ for $k=1$ and $2$ are only known in limited cases. See \cite{Ha00} for another example on Poisson processes. 

\begin{table}
\begin{center}
\begin{tabular}{| l | l | l | l | l | l | l}
\hline
$n$ & $t$ & $\mu_0$ & $a$ & $p_1$ & $p_2$ \\ \hline
7680 & 30 & 0.1 & 11/30 & 0.14097 & 0.14021\\ \hline
7680 & 30 & 0.1 & 0.4 & 0.029614 & 0.029387\\ \hline
15360 & 30 & 0.1 & 0.4 & 0.058458 & 0.058003\\ \hline
\end{tabular}
\end{center}
\caption{Example 2.2}
\label{table2}
\end{table}

\begin{remark}
From the proof of Theorem \ref{t1}, the $\lambda$ in Example \ref{example2} can be reduced to
\be{
\lambda=\frac{ (n-t+1) e^{-(a\theta_a-\Psi(\theta_a))t} e^{-\theta_a (\lceil at\rceil -at)}}
{[2a(1-a)\pi t]^{1/2}}\times (a-\mu_0).
}
The reason is that the intermediate quantity $\lambda_2$ in \eq{101} for Bernoulli random variables can be expressed as
\be{
\lambda_2=(n-t+1) \P(T_1=\ceil{at}) \P(D_i>0, i\geq 1),
}
and 
\bes{
\P(D_i>0, i\geq 1)&=\P(D_1>0)\times \P(D_i\geq 0, i\geq 1)\\
&=a(1-\mu_0)\times [1-\frac{\mu_0(1-a)}{a(1-\mu_0)}]\\
&=a-\mu_0,
}
where the second equation is from a known result for the first visit to $-1$ of a Bernoulli random walk starting from $0$ (see, e.g., page 272 of \cite{Fe68}).
\end{remark}
\end{example}

The following corollary considers the case that $X_i$ is integer-valued and $a$ is the largest value $X_i$ can take. The proof of it, which is deferred to Section 3, is simpler than the proof of Theorem \ref{t1} and the convergence rate we obtain is faster.
\begin{corollary}\label{c1}
Let $\{X_1,\dots, X_n\}$ be independent, identically distributed integer-valued random variables. For integers $t< n$,
define $M_{n; t}$ as in \eq{0.1}. Suppose $a=\sup\{x: p_x:=\P(X_1=x)>0\}$ is finite.
We have, with constants $C$ and $c$ depending only on $p_a$,
\ben{\label{c1-1}
\big| \P(M_{n; t}\geq at) - (1-e^{-\lambda}) \big|\leq C(\lambda \wedge 1) e^{-ct}
}
where
\be{
\lambda=(n-t) p_a^t (1- p_a )+p_a^t.
}
\end{corollary}

\subsection{Scan statistics with varying window size}

Next we study the maximum log likelihood ratio statistic \eq{0.2}.
Suppose in \eq{0.3}, $f_0(x)=dF_{\theta_0}(x)$ and $f_1(x)=dF_{\theta_1}(x)$ where $\{F_\theta: \theta\in \Theta\}$ is an exponential family as in \eq{1.2} and $\theta_0<\theta_1$. Then we have
\be{
S_i=\sum_{k=1}^i \log[f_1(X_k)/f_0(X_k)]=\sum_{k=1}^i (\theta_1-\theta_0)  \big( X_k-\frac{\Psi(\theta_1)-\Psi(\theta_0)}{\theta_1-\theta_0} \big).
}
By appropriate change of parameters and a slight abuse of notation, studying \eq{0.2} is equivalent to studying the following problem.

Let $\{X_1,\dots, X_n\}$ be independent, identically
distributed  random variables with distribution function $F$ that
can be imbedded in an exponential family, as in \eq{1.2}.  Let  $\E X_1=\mu_0<0$. 
Let $S_0=0$ and
$S_i=\sum_{j=1}^i X_j$ for $1\leq i\leq n$.  
We are interested in the distribution of $\max_{0 \leq i < j \leq n} (S_j-S_i).$
Statistics of this form have been widely studied in the context of CUSUM
tests.  Its limiting distribution was derived by \cite{Ig72}, who 
observed that it can be interpreted
as the maximum waiting time of the first $n$ customers in a single server queue.
Genomic applications are  discussed by \cite{KaDeKa90}.

Suppose there exists a positive $\theta_1\in \Theta$ such that
\ben{\label{2.2}
 \Psi'(\theta_1)=\mu_1,\quad \Psi(\theta_1)=0.
}
For $b>0$, in the following theorem we give an approximation to
\ben{\label{2.1}
p_{n,b}:=\IP \big( \max_{0 \leq i<j\leq n}(S_j-S_i)\geq b \big)
}
with an explicit error bound. 
We again consider two cases:
\begin{itemize}
\item[Case 1:] The distribution $F_{\theta_1}$ satisfies
$\int_{-\infty}^\infty |\varphi_{\theta_1} (t)|dt<\infty$.

\item[Case 2:] $X_1$ is an integer-valued random variable not concentrated on the set $\{jd, -\infty<j<\infty\}$ for any $d>1$.
\end{itemize}
In the following, let $\P_\theta(\cdot)$ ($\E_\theta(\cdot)$ resp.) denote the probability (expectation resp.) under which $\{X_1, X_2, \dots\}$ are independent, identically distributed as $F_\theta$.

\begin{theorem}\label{t2}
Let $h(b)>0$ be any function such that
\be{
h(b)\rightarrow \infty, \ h(b)=O(b^{1/2}) \  \text{as}\ b\rightarrow \infty.
}
Suppose $n-b/\mu_1>b^{1/2}h(b)$. Under the above setting, we have, for some constants $c, C$ only depending on the exponential family $F_\theta$ and $\theta_1$,
\ben{\label{t2-1}
\big| p_{n,b}-(1-e^{-\lambda}) \big|\leq C \lambda  \left\{ \Big( 1+\frac{b/h^2(b)}{n-b/\mu_1}  \Big) e^{-c h^2(b)} 
+ \frac{b^{1/2}h(b)}{n-\frac{b}{\mu_1}} \right\}
}
where for Case 1,
\be{
\lambda=(n-\frac{b}{\mu_1})   \frac{e^{-\theta_1 b}}{\theta_1 \mu_1} \exp( -2\sum_{k=1}^{\infty} \frac{1}{k} \E_{\theta_1} e^{-\theta_1  S_k^+ } )  ,
}
and for Case 2 and integers b,
\be{
\lambda=(n-\frac{b}{\mu_1}) \frac{e^{-\theta_1 b}}{(1-e^{-\theta_1})\mu_1} \exp( -2\sum_{k=1}^{\infty} \frac{1}{k} \E_{\theta_1} e^{-\theta_1 S_k^+ } ).
}
\end{theorem}
\begin{remark}
We refer to Remark \ref{r3} for the numerical calculation of $\lambda$.
By choosing $h(b)=b^{1/2}$, we get
\be{
|p_{n,b}-(1-e^{-\lambda})|\leq C\lambda \{e^{-cb} + \frac{b}{n}\}
}
from \eq{t2-1}. By choosing $h(b)=C(\log b)^{1/2}$ with large enough $C$, we can see that
the relative error in the Poisson approximation goes to zero under the conditions
\be{
b\rightarrow \infty,\quad (b \log b)^{1/2}  \ll n-b/\mu_1 =O(e^{\theta_1 b}),
}
where $n-b/\mu_1 =O(e^{\theta_1 b})$ ensures that $\lambda$ is bounded.
For the smaller range (in which case $\lambda \rightarrow 0$)
\be{
b\rightarrow \infty, \quad \delta b\leq n-b/\mu_1=o(e^{\frac{1}{2} \theta_1 b})
}
for some $\delta>0$, Theorem 2 of \cite{Si88} obtained more accurate estimates by a technique different from ours.
\end{remark}

As in the case of Theorem 2.1, in some simple cases there is also the possibility here to
evaluate $\lambda$ by direct argument and hence avoid the
need for the numerical calculations of the general theory.  Suppose
the $X_j$ are integer valued with either the maximum of the support
equal to 1 or the minimum of the support equal to $-1$.
Two interesting examples mentioned explicitly in \cite{KaDeKa90}
involve these possibilities.  For example, assume that $X_i$ equals
$k \geq 0$ with probability $p_k$ and the negative value $-k$ with probability
$q_k$.  Let $G(z) = \sum_0^\infty p_k z^k + \sum_1^\infty q_k z^{-k}$,
and let $z_0$ denote the unique root $ > 1$ of $G(z) = 1$.
For the case $p_k = 0$ for $k > 1$, using the notation
$Q(z) = \sum_k q_k z^k$, one can
show for large values of $n$ and $b$
that $\lambda \sim n z_0^{-b} \{[Q(1) - Q(z_0^{-1})]- (1-z_0^{-1}) z_0^{-1} Q'(z_0^{-1})\}$.
For the case $q_k = 0$ for $k > 1$,
$\lambda \sim n z_0^{-b}(1 - z_0^{-1}))|G'(1)|^2/G'(z_0).$
In particular if $q_1 = q$ and $p_1 = p$, where $p+q = 1$, both
these results specialize to
$\lambda \sim n (p/q)^b (q-p)^2/q.$
These results differ from those given in \cite{KaDeKa90} and hence produce
numerical results slightly different from those cited there.

\section{Proofs}

Before proving our main theorems, we first introduce our main tool: Stein's method.
Stein's method was first introduced by \cite{St72} and further developed in \cite{St86} for normal approximation. \cite{Ch75} developed Stein's method for Poisson approximation, which has been widely applied especially in computational biology after the work by \cite{ArGoGo90}. We refer to \cite{BaCh05} for an introduction to Stein's method.

The following theorem provides a useful upper bound on the total variation distance between the distribution of a sum of locally dependent Bernoulli random variables and a Poisson distribution.
The total variation distance between two distributions is defined as
\be{
d_{TV}(\mathcal{L}(X), \mathcal{L}(Y))=\sup_{A\subset \mathbb{R}}|\P(X\in A)-\P(Y\in A)|.
}

\begin{theorem}[\cite{ArGoGo90}]\label{t3} 
Let $W=\sum_{\alpha\in A} Y_\alpha$ be a sum of Bernoulli random variables where $A$ is the index set and $\P(Y_\alpha=1)=1-\P(Y_\alpha=0)=p_\alpha$. Let $\lambda=\sum_{\alpha\in A} p_\alpha$, and let $Poi(\lambda)$ denote the Poisson distribution with mean $\lambda$. Then,
\ben{\label{t3-1}
d_{TV}(\mathcal{L}(W), Poi(\lambda))\leq (1\wedge \frac{1}{\lambda}) (b_1+b_2+b_3)
}
where for each $\alpha$ and $B_\alpha$ such that $\alpha\in B_\alpha\subset A$,
\besn{\label{t3-2}
&b_1:=\sum_{\alpha\in A}\sum_{\beta\in B_\alpha} p_\alpha p_\beta,\\
&b_2:=\sum_{\alpha\in A} \sum_{\alpha\ne \beta \in B_\alpha} \E (Y_\alpha Y_\beta),\\
&b_3:=\sum_{\alpha\in A} \E \Big| \E \big(Y_\alpha-p_\alpha \big| \sigma(Y_\beta: \beta\notin B_\alpha)  \big)  \Big|.
}
\end{theorem}
\begin{remark}
If $B_\alpha$ is chosen such that $X_\alpha$ is independent of $\{X_\beta: \beta\notin B_\alpha\}$, then $b_3$ in \eq{t3-1} equals $0$. Roughly speaking, in order for $b_1$ and $b_2$ to be small, the size of $B_\alpha$ has to be small and $\E(Y_\beta|Y_\alpha=1)=o(1)$ for $\alpha\ne \beta \in B_\alpha$.
\end{remark}

\subsection{Proof of Theorem \ref{t1}}
In this subsection, let $C$ and $c$ denote positive constants depending only on the exponential family \eq{1.2}, 
$\mu_0$, and $a$. They may represent different values in different expressions.
The lemmas used in the proof of Theorem \ref{t1} will be stated and proved after the proof of the Theorem.
\begin{proof}[Proof of Theorem \ref{t1}]
By the union bound and Lemma \ref{lemma1}, we have
\besn{\label{star5}
\P(M_{n;t}\geq at)&\leq (n-t+1)\P(X_1+\dots + X_t\geq at)\\
&\sim (n-t+1) e^{-(a\theta_a -\Psi(\theta_a))t} \big/ t^{1/2}
}
where $x\sim y$ means that $x/y$ is bounded away from zero and infinity.
On the other hand, by the definition of $\lambda$ in \eq{t1-2} and \eq{t1-3}, we have
\ben{\label{star101}
\lambda\sim (n-t+1) e^{-(a\theta_a -\Psi(\theta_a))t} \big/ t^{1/2}.
}
From \eq{star5} and \eq{star101}, if $t$ or $n-t$ is bounded, then the bound \eq{t1-1} holds true by choosing $C$ in \eq{t1-1} to be large enough.
Therefore, in the sequel, we can assume $t$ and $n-t$ to be larger than any given constant.

Let $\delta$ be a positive number such that 
\ben{\label{1.19}
0<\delta<1\wedge (a-\mu_0)/4 \quad \text{and} \quad \Psi(\theta_a)-(\mu_0+\delta)\theta_a >0.
}
The second inequality above is possible because of the strict convexity of $\Psi$.
Let
\ben{\label{1.20}
m= \lfloor C (\log t \wedge \log(n-t)) \rfloor
}
where the constant $C$ will be chosen later in \eq{801}.
By Lemma \ref{lemma2}, we can find $\theta_{a_1}\in \Theta$ such that $\theta_a<\theta_{a_1}$,
$\theta_{a_1}\leq \theta_{a'}$ for Case 1
and for $t$ and $n-t$ larger than some unspecified constant, the following bound holds uniformly in $a_2\in [a, a_1]$:
\ben{\label{719}
d_{TV}\Big(\mathcal{L}\big( \widetilde{X}_i^{a_2}: 1\leq i\leq m \big), \mathcal{L}\big( X_i^{a_2}: 1\leq i\leq m \big)    \Big)\leq C\frac{m}{t},
}
where $\{\widetilde{X}_1^{a_2},\dots, \widetilde{X}_m^{a_2}\}$ is distributed as the conditional distribution of $\{X_1,\dots, X_m\}$ given $X_1+\dots+X_t=a_2 t$, and $\{X_1^{a_2},\dots, X_m^{a_2}\}$ are independent, identically distributed as $F_{\theta_{a_2}}$.
In the following, we fix such an $a_1$ and assume $t$ and $n-t$ to be large enough so that the bound \eq{719} holds and
\ben{\label{1.21}
m<(a_1-a)t/\delta.
}
We embed the sequence $\{X_1,\dots, X_n\}$ into an infinite i.i.d. sequence 
\be{
\{\dots,X_{-1},X_0,X_1,\dots\}.
}
For each integer $\alpha$, let
\ben{\label{709}
T_\alpha=X_\alpha +\dots + X_{\alpha+t-1},\quad \widetilde{Y}_\alpha=I(T_\alpha \geq at ).
}
To avoid the clumping of $1$'s in the sequence $(\widetilde{Y}_\alpha)$ which makes a Poisson approximation invalid, we define
\ben{\label{710}
Y_\alpha=I(T_\alpha\geq at, T_{\alpha-1}<at,\dots,T_{\alpha-m}<at).
}
Let
\ben{\label{704}
W=\sum_{\alpha=1}^{n-t+1} Y_\alpha, \quad \lambda_1=\E W=(n-t+1) \E Y_1.
}
In the following, we first bound $| \P(M_{n;t}\geq at)-\P(W\geq 1) |$, then bound the total variation distance between the distribution of $W$ and $Poi(\lambda_1)$, finally we bound $|\lambda_1-\lambda|$.

First, since $\{M_{n;t}\geq at\}\backslash \{W\geq 1\}\subset \cup_{\alpha=1}^m \{T_\alpha\geq at\}$,
we have
\ben{\label{1.1}
0\leq \P(M_{n;t}\geq at)-\P(W\geq 1) \leq m \P(T_1 \geq  at).
}
Next, by applying Theorem \ref{t3}, we prove in Lemma \ref{lemma3} that
\besn{\label{1.11}
&\big| \P(W\geq 1)-(1-e^{-\lambda_1})   \big| \\
&\leq C(1\wedge \frac{1}{\lambda_1}) (n-t+1) \P(T_1\geq at)  [t \P(T_1\geq at)+e^{-cm}],
}
where the constant $c$ does not depend on the choice of the constant $C$ in \eq{1.20}, as can be seen from the proof of Lemma \ref{lemma3}.
Since $\lambda_1$ does not have an explicit expression, our final goal is to show that $\lambda_1$ is close to $\lambda$, which can be calculated explicitly as discussed in Remark \ref{r3}. For this purpose, we first introduce an intermediate quantity $\lambda_2$ defined as
\ben{\label{706}
\lambda_2=(n-t+1)\int_{at}^{\infty} \P(D_i>s-at, i\geq 1)d\P(T_1\leq s).
}
Lemma \ref{lemma4} shows that 
\ben{\label{star3}
|\lambda_1 - \lambda_2|\leq C (n-t) \big[ \frac{m^2}{t} + e^{-c m} \big] \P(T_1\geq at),
}
and Lemma \ref{lemma5} shows that
\ben{\label{714}
\lambda_2=\lambda \big(1+O(\frac{(\log t)^2}{t}) \big).
}
Again, from the proof of Lemma \ref{lemma4}, the constant $c$ in \eq{star3} does not depend on the choice of the constant $C$ in \eq{1.20}. 
Let the constant $C$ in \eq{1.20} be chosen such that
\ben{\label{801}
e^{-cm}=O(\frac{1}{t} \vee \frac{1}{n-t})
}
for the constants $c$ in \eq{1.11} and \eq{star3}.
By Lemma \ref{lemma1} and \eq{star101}, we have
\ben{\label{802}
\lambda\sim (n-t)\P(T_1\geq at).
}
This, together with \eq{star3}, \eq{714} and \eq{1.20}, implies
\ben{\label{803}
|\lambda_1-\lambda|\leq C\lambda \big[ \frac{(\log t)^2}{t} +\frac{1}{n-t} \big] \quad \text{and}\quad  \lambda_1\sim \lambda.
}
By \eq{1.1} and \eq{1.11}, we have
\bes{
&|\P(M_{n;t}\geq at)-(1-e^{-\lambda_1})|\\
\leq& C(1\wedge \frac{1}{\lambda_1}) (n-t) \P(T_1\geq at)
\Big[  t\P(T_1\geq at) +e^{-cm}+\frac{m}{n-t}   \Big],
}
where the constant $c$ is the same as that in \eq{1.11}.
By \eq{802}, \eq{803}, \eq{star5}, \eq{801} and \eq{1.20}, this is further bounded by
\ben{\label{804}
|\P(M_{n;t}\geq at)-(1-e^{-\lambda_1})|\leq C(\lambda \wedge 1) \Big[ \frac{1}{t}+\frac{\log t\wedge \log(n-t)}{n-t}  \Big].
}
The bound \eq{t1-1} is proved by using \eq{803} and \eq{804} for the cases $\lambda\leq 1$ and $\lambda> 1$ separately and using $|e^{-\lambda}-e^{-\lambda_1}|\leq |\lambda-\lambda_1| e^{-(\lambda\wedge \lambda_1)}$.
\end{proof}

The following lemmas have been used in the above proof.

\begin{lemma}[Theorem 1 and Theorem 6 of \cite{Pe65}]\label{lemma1}
Under the setting of Theorem \ref{t1}, we have
\be{
\P(X_1+\dots + X_t\geq at)\sim e^{-(a\theta_a -\Psi(\theta_a))t} \big/ t^{1/2}.
}
\end{lemma}

\begin{lemma}\label{lemma2}
Under the setting of Theorem \ref{t1}, there exists $\theta_{a_1}\in \Theta$ such that $\theta_a<\theta_{a_1}$,
$\theta_{a_1}\leq \theta_{a'}$ for Case 1 and the bound \eq{719} holds uniformly in $a_2\in [a, a_1]$ and in $m$ and $t$ such that $m$, $t$ and $t/m$ are larger than some unspecified constant.
\end{lemma}
\begin{proof}
For Case 1, by \eq{702}, we have $|\varphi_{\theta_{a}}(x)|<1$ for $x\ne 0$ and $|\varphi_{\theta_{a}}(x)| \rightarrow 0$ as $|x| \rightarrow \infty$. Therefore, there exists $M>0$ such that $|\varphi_{\theta_a}(x)|<1/2$ for $|x|>M$. By the dominated convergence theorem,
\be{
|\varphi_{\theta_a+h}(x)-\varphi_{\theta_a}(x)| \rightarrow 0\ \text{as} \ h\rightarrow 0^+ \ \text{uniformly in $x$}.
}
This, together with the continuity of the function $\varphi_{\theta}(\cdot)$, implies that there exists $a_1\leq a'$ such that
\ben{\label{703}
\sup_{\theta_a\leq \theta\leq \theta_{a_1}}\sup_{|x|>\epsilon}|\varphi_{\theta}(x)|<1\ \text{for all}\ \epsilon>0.
}
We now show that with such choice of $a_1$, \eq{719} is satisfied.
We follow the proof of Theorem 1.6 of \cite{DiFr88}. Since only the range of parameters $[a,a_1]$ enters into considerations, we do not need their Condition 1.1. By \eq{702} and \eq{703}, their Conditions 1.2--1.4 are satisfied for the range of parameters $[a,a_1]$.
Following their proof of Theorem 1.6, \eq{719} holds uniformly in $a_2\in [a, a_1]$ and in $m$ and $t$ such that $m$, $t$ and $t/m$ are larger than some unspecified constant.
The bound \eq{719} for Case 2 can be proved by rewriting the proof (e.g., changing density functions to probability mass functions) for Case 1. As mentioned in \cite{DiFr88}, the proof for the discrete case is a little easier. Therefore, we omit the details here.
\end{proof}

\begin{lemma}\label{lemma3}
Under the setting of Theorem \ref{t1}, we have
\bes{
&\big| \P(W\geq 1)-(1-e^{-\lambda_1})   \big| \\
&\leq C(1\wedge \frac{1}{\lambda_1}) (n-t+1) \P(T_1\geq at)  [t \P(T_1\geq at)+e^{-cm}].
}
where $W$ and $\lambda_1$ are defined in \eq{704}, $T_1$ is defined in \eq{709}, and $m$ is defined in \eq{1.20}.
\end{lemma}
\begin{proof}
We apply Theorem \ref{t3} to bound the total variation distance between the distribution of $W$ and $Poi(\lambda_1)$. For each $1\leq \alpha \leq n-t+1$, define $B_\alpha=\{ 1\leq \beta \leq n-t+1: |\alpha-\beta|<t+m \}$. By definition of $B_\alpha$, $Y_\alpha$ in \eq{710} is independent of $\{Y_{\beta}: \beta\notin B_\alpha\}$. Therefore, $b_3$ in \eq{t3-2} equals zero. Since $|B_\alpha|< 2(t+m)$,
\be{
b_1=\sum_{1\leq \alpha \leq n-t+1}\sum_{\beta\in B_\alpha} \E Y_\alpha \E Y_\beta < 2(t+m)\lambda_1 \E Y_1.
}
By the definition of $Y_\alpha$, for $1\leq |\beta-\alpha|\leq m$, $\E Y_\alpha  Y_\beta=0$, and for $m<|\beta-\alpha|<t+m$, $\E Y_\alpha  Y_\beta \leq \E \widetilde{Y}_{\alpha\wedge \beta}  \widetilde{Y}_{\alpha\vee \beta}$
where $\widetilde{Y}_\alpha$ is defined in \eq{709}. Therefore, by symmetry,
\be{
b_2=\sum_{1\leq \alpha \leq n-t+1}\sum_{\alpha\ne \beta \in B_\alpha} \E Y_\alpha  Y_\beta<2(n-t+1)
\E \widetilde{Y}_1 \sum_{\beta=m+2}^{m+t} \P(T_\beta\geq at |T_1 \geq at).
}
For $\beta\geq t+1$, by independence,
\be{
 \P(T_\beta\geq at|T_1\geq at)= \P(T_1\geq at).
}
Let $\delta$ be the positive number defined above \eq{1.19} such that $\eq{1.19}$ is satisfied, and let $a_1$ be as in Lemma \ref{lemma2}.
We observe that for $m+2\leq \beta\leq t$, $T_\beta\geq at$ and $X_{t+1}+\dots +X_{t+\beta-1}\leq (\mu_0+\delta)(\beta-1)$ together imply $X_\beta+\dots + X_t\geq at -(\mu_0+\delta)(\beta-1)$. Therefore,
\bes{
&\sum_{\beta=m+2}^{t} \P(T_\beta \geq at|T_1\geq at)\\
&\leq \sum_{\beta=m+2}^{t} \big\{ \P(X_{t+1}+\dots+X_{t+\beta-1}>(\mu_0+\delta)(\beta-1)) \\
&\kern4em +\P \big( X_\beta+\dots +X_{t} \geq at -(\mu_0+\delta)(\beta-1)   \big| T_1\geq at \big)  \big\}.
}
For the first term, we have
\besn{\label{1.14}
&\sum_{\beta=m+2}^{t} \P(X_{t+1}+\dots +X_{t+\beta-1}>(\mu_0+\delta)(\beta-1))\\
&\leq \sum_{\beta=m+2}^{t} e^{-[\theta_{\mu_0+\delta}(\mu_0+\delta)-\Psi(\theta_{\mu_0+\delta})](\beta-1)}\\
&\leq \frac{e^{-[\theta_{\mu_0+\delta}(\mu_0+2\delta)-\Psi(\theta_{\mu_0+\delta})]m}}{1-e^{-[\theta_{\mu_0+\delta}(\mu_0+2\delta)-\Psi(\theta_{\mu_0+\delta})]}}.
}
By the bound on $V$ on page 613 of \cite{KoTu75}
and recalling that we have chosen $\delta$ such that $\Psi(\theta_a)-(\mu_0+\delta)\theta_a >0$, we have
\besn{\label{1.15}
&\sum_{\beta=m+2}^{t} \P \big( X_\beta+\dots +X_{t} \geq at - (\mu_0+\delta)(\beta -1)   \big| T_1\geq at \big)\\
&\leq C \sum_{\beta=m+2}^{t}  e^{-[\Psi(\theta_a)-(\mu_0+\delta)\theta_a](\beta-1)}
\sqrt{\frac{t}{t-\beta+1}}  \\
&\leq C \frac{e^{-[\Psi(\theta_a)-(\mu_0+\delta)\theta_a]m}}{(1- e^{-[\Psi(\theta_a)-(\mu_0+\delta)\theta_a]})}.                                                                                                                                                                                                                                                                                                        
}
Therefore,
\be{
b_2\leq C(n-t+1)\P(T_1\geq at) [m \P(T_1\geq at)+e^{-cm}].
}
Lemma \ref{lemma3} is then followed by Theorem \ref{t3} and the above bounds on $b_1$ and $b_2$.
\end{proof}

\begin{lemma}\label{lemma4}
Under the setting of Theorem \ref{t1}, we have
\be{
|\lambda_1 - \lambda_2|\leq C (n-t) \big[ \frac{m^2}{t} + e^{-c m} \big] \P(T_1\geq at)
}
where $\lambda_1$ and $\lambda_2$ are defined in \eq{704} and \eq{706}, $m$ is defined in \eq{1.20} and satisfies \eq{1.21}, and $T_1$ is defined in \eq{709}.
\end{lemma}
\begin{proof}
By symmetry, we can write
\besn{\label{1.3}
&\E Y_1= \I(T_1\geq at, T_2<at,\dots, T_{m+1}<at)\\
&= \E \widetilde{Y}_1 (1-\widetilde{Y}_2)\dots(1-\widetilde{Y}_{m+1})\\
&\leq \int_{at}^{at+m \delta} \E \big[  (1-\widetilde{Y}_2)\dots(1-\widetilde{Y}_{m+1})   \big| T_1=s \big]            d \P(T_1\leq s)     + \P(T_1>at+m\delta)
}
where $\widetilde{Y}_\alpha$ is defined in \eq{709} and $\delta$ is the positive number defined above \eq{1.19} such that \eq{1.19} is satisfied. Observe that $T_1=s$ and $T_{i+1}<at$ imply $T_1-T_{i+1}=S_i-(S_{i+t}-S_t)>s-at$ where $S_i=\sum_{j=1}^i X_j$. Therefore, given $T_1=s$, $(1-\widetilde{Y}_2)\dots(1-\widetilde{Y}_{m+1})$ is the indicator of the event that 
$\{ \widetilde{S}_i^{s/t}-S_i>s-at, 1\leq i\leq m\}$ where $\widetilde{S}_i^{s/t}$ is independent of $S_i$,
\be{
\widetilde{S}_i^{s/t}=\sum_{j=1}^i  \widetilde{X}_j^{s/t}\quad \text{and} \quad 
\mathcal{L}\big( \widetilde{X}_i^{s/t}: 1\leq i\leq m \big)=\mathcal{L}\big(X_i: 1\leq i\leq m \big| S_t=s \big).
}
Note that the assumption $m<(a_1-a)t/\delta$ in \eq{1.21} implies $a\leq s/t\leq a_1$ for $s\in [at, at+m\delta]$.

By the definition of total variation distance, for $at\leq s\leq at+m\delta$ and $m\geq \nu$,
\besn{\label{806}
&d_{TV}\Big(\mathcal{L}\big( X_i^{s/t}: 1\leq i\leq m \big), \mathcal{L}\big( X_i^a: 1\leq i\leq m \big)    \Big)\\
&\leq \E_{\theta_{s/t}}I(S_{m}>t/m) + \E_{\theta_a} I(S_{m}>t/m)\\
&\quad+\E_{\theta_a} \big| e^{(\theta_{s/t}-\theta_{a})S_{m}-m(\Psi(\theta_{s/t})-\Psi(\theta_a))} -1 \big|I(S_{m}\leq t/m).
}

For $s/t\in [a,a_1]$ and $s/t-a\leq m\delta/t$, we have
\besn{\label{sstar1}
&|\theta_{s/t}-\theta_a|\leq \sup_{\theta_a\leq \theta\leq \theta_{a_1}}\frac{1}{\Psi''(\theta)}(s/t-a)\leq \sup_{\theta_a\leq \theta\leq \theta_{a_1}}\frac{m\delta}{\Psi''(\theta)t} ,\\
&|\Psi(\theta_{s/t})-\Psi(\theta_a)|\leq \sup_{\theta_a\leq \theta\leq \theta_{a_1}} |\Psi'(\theta)|(s/t-a)\leq \sup_{\theta_a\leq \theta\leq \theta_{a_1}} |\Psi'(\theta)| \frac{m\delta}{t}.
}
This implies that if $a\leq s/t\leq a+m\delta/t$, $S_m\leq t/m$ and $m\leq \sqrt{t}$, then
\ben{\label{807}
(\theta_{s/t}-\theta_{a})S_{m}-m(\Psi(\theta_{s/t})-\Psi(\theta_a))\leq C.
}
By \eq{806}, Markov's inequality, \eq{807} and the fact that $|e^t-1|\leq Ct$ for bounded $t$, we have, for $at\leq s\leq at+m\delta$,
\besn{\label{sstar2}
&d_{TV}\Big(\mathcal{L}\big( X_i^{s/t}: 1\leq i\leq m \big), \mathcal{L}\big( X_i^a: 1\leq i\leq m \big)    \Big)\\
&\leq \frac{m}{t}(\E_{\theta_{s/t}}|S_m|+\E_{\theta_a}|S_m|)+C\E_{\theta_a}|(\theta_{s/t}-\theta_{a})S_{m}-m(\Psi(\theta_{s/t})-\Psi(\theta_a))  |\\
&\leq Cm^2/t
}
where in the last inequality we used $s/t\in [a,a_1]$, $\E_{\theta} |S_m|\leq Cm$ for $\theta\in [\theta_a, \theta_{a_1}]$ and \eq{sstar1}.
By \eq{1.3} and the argument just below it, \eq{719} and \eq{sstar2}, we have
\besn{\label{1.7}
&\E Y_1 \leq \int_{at}^{at+m\delta} \big[ \P(D_i>s-at, 1\leq i\leq m)    +C\frac{m^2}{t} \big]        d \P(T_1\leq s) \\
&\kern3em    + \P(T_1>at+m\delta).
}
By \eq{1.19}, $at\leq s\leq at+m\delta$ implies
\ben{\label{712}
s-at-m(a-\mu_0)/2<m(\mu_0-a)/4.
}
From the FKG inequality (cf. (1.7) of \cite{KaRi80}) and the fact that $\I(D_i>s-at, 1\leq i\leq m)$ and $\I(D_i>s-at, m_1\geq i>m)$ are both increasing functions of $\{X_1^a-X_1, \dots, X_{m_1}^a-X_{m_1}\}$ for $m_1>m$,
we have
\besn{\label{1.4}
& \P(D_i> s-at, m_1\geq i\geq 1)\\
& = \P(D_i> s-at, 1\leq i\leq m) \P(D_i> s-at, m_1\geq i> m |D_i> s-at, 1\leq i\leq m )\\
& \geq \P(D_i> s-at, 1\leq i\leq m) \P( D_i> s-at, m_1\geq i> m )\\
& \geq \P(D_i> s-at, 1\leq i\leq m) \P\big(D_i> s-at, i> m , D_{m} \geq m(a-\mu_0)/2\big)\\
& \geq \P(D_i> s-at, 1\leq i\leq m)\\
&\quad \times \Big\{  1-\P\big( D_{m} <m(a-\mu_0)/2 \big) - \sum_{i=1}^\infty \P\big( D_i< m(\mu_0-a)/4   \big) \Big\},
}
where the last inequality in \eq{1.4} follows from \eq{712}.
Letting $m_1 \to \infty$ in \eq{1.4}, we have
\besn{\label{1.41}
& \P(D_i> s-at, i\geq 1)\\
& \geq \P(D_i> s-at, 1\leq i\leq m)\\
&\quad \times \Big\{  1-\P\big( D_{m} <m(a-\mu_0)/2 \big) - \sum_{i=1}^\infty \P\big( D_i< m(\mu_0-a)/4   \big) \Big\}.
}
For $0<r\leq \theta_a$, we have
\be{
\E \exp(-r D_i)=\exp \big\{ -i\big[ \Psi(\theta_a)-\Psi(r)-\Psi(\theta_a -r)   \big]   \big\}.
}
By Taylor's expansion,
\be{
\Psi(\theta_a)-\Psi(\theta_a-r)=ra-\frac{r^2}{2} \Psi''(\theta_a-r_1), \quad 
-\Psi(r)=-r\mu_0-\frac{r^2}{2} \Psi''(r_2)
}
where $0\leq r_1, r_2\leq r$. Therefore,
\besn{\label{808}
&\P\big( D_{m} <m(a-\mu_0)/2\big)\\
&\leq \exp \big\{ -m\big[ \Psi(\theta_a)-\Psi(r)-\Psi(\theta_a -r)  - \frac{(a-\mu_0)r}{2}  \big]   \big\} \\
&=\exp \big\{ -m\big[ \frac{r}{2}(a-\mu_0) -\frac{r^2}{2} \big( \Psi''(\theta_a-r_1) +\Psi''(r_2) \big) \big]  \big\}.
}
Let
\be{
c_1=\frac{a-\mu_0}{4\theta_a }\vee  \max_{\theta\in \Theta: 0\leq \theta \leq \theta_{a}} \Psi''(\theta).
}
Choosing $r=(a-\mu_0)/(4c_1)$ in \eq{808}, we have
\ben{\label{1.5}
\P\big( D_{m} <m(a-\mu_0)/2\big)\leq \exp \big\{ - \frac{(a-\mu_0)^2}{16c_1} m  \big\}.
}
Similarly,
\ben{\label{1.6}
\P\big( D_i< m(\mu_0-a)/4) \leq \exp\big\{ - \frac{(a-\mu_0)^2}{16c_1}  i -\frac{(a-\mu_0)^2}{16c_1} m \big\}.
}
Applying \eq{1.5} and \eq{1.6} in \eq{1.41}, we obtain
\bes{
&\P(D_i> s-at, 1\leq i\leq m)\\
& \leq 
\P(D_i> s-at, i\geq 1)+2 \exp \big\{ - \frac{(a-\mu_0)^2}{16c_1} m  \big\} \big/ \big( 1- \exp \big\{ - \frac{(a-\mu_0)^2}{16c_1}   \big\} \big).
}
Therefore, by \eq{1.7},
\besn{\label{1.8}
&\E Y_1-\lambda_2/(n-t+1)\\
& \leq  \bigg\{ \frac{2\exp\big\{ -m(a-\mu_0)^2/(16c_1)  \big\}}{1-\exp\big\{-(a-\mu_0)^2/(16c_1) \big\}} +C \frac{m^2}{t} +\frac{\P(T_1>at+m\delta)}{\P(T_1\geq at)}  \bigg\} \P(T_1\geq at).
}
From the corollary on page 611 of \cite{KoTu75}, and recalling that in proving \eq{t1-1}, we can only consider those $t$ larger than any given constant, we have
\ben{\label{1.9}
 \P(T_1>at+m\delta |T_1\geq at)\leq C e^{-\theta_a m\delta}.
}
After proving a similar and easier lower bound of $\E Y_1$, we obtain Lemma \ref{lemma4}.
\end{proof}

\begin{lemma}\label{lemma5}
Under the setting of Theorem \ref{t1}, we have
\be{
\lambda_2=\lambda \big(1+O(\frac{(\log t)^2}{t}) \big)
}
where $\lambda_2$ is define in \eq{706}.
\end{lemma}
\begin{proof}
We first consider Case 1 of Theorem \ref{t1}.
By the proof of Theorem 2.7 of \cite{Wo82}, we have for $x\geq 0$,
\ben{\label{1.16}
\P(D_i> x, i\geq 1)=\frac{\P (D_{\tau_+}>x)}{\E \tau_{+}} 
}
where
$\tau_{+}=\inf\{i\geq 1, D_i> 0\}$.
Let
$x_0=\log t /\theta_a$.
By change of variable and \eq{1.2}, we have
\besn{\label{101}
\lambda_2&=(n-t+1)\int_0^{\infty} \P(D_i>x, i\geq 1) d\P (T_1\leq at+x)\\
&=(n-t+1) e^{-(a\theta_a -\Psi(\theta_a))t} \\
&\quad \times \int_0^{x_0} \P(D_i>x, i\geq 1)e^{-\theta_a x} d\P_{\theta_{a}} (T_1\leq at+x)\\
&\quad +O((n-t+1) \P (T_1>at+x_0)),
}
where $T_\alpha$ is defined in \eq{709}.
By the local central limit theorem (cf. \cite{Fe71}), uniformly for $0\leq x\leq x_0$,
\ben{\label{1.17}
d\P_{\theta_{a}} (T_1\leq at+x)=\frac{1}{\sigma_a (2\pi t)^{1/2}}+O(\frac{(\log t)^2}{t^{3/2}}).
}
From \eq{1.9} and Lemma \ref{lemma1}, we have
\besn{\label{sstar3}
&\P(T_1>at+x_0)=\P(T_1>at+x_0|T_1\geq at)\P(T_1\geq at)\\
&\leq C e^{-\theta_a x_0}\frac{e^{-(a\theta_a-\Psi(\theta_a))t}}{\sqrt{t}}
\leq C\frac{e^{-(a\theta_a-\Psi(\theta_a))t}}{\sqrt{t}} \frac{1}{t}.
}

Applying \eq{1.16}, \eq{1.17} and \eq{sstar3} in \eq{101}, we obtain
\bes{
\lambda_2&=\frac{(n-t+1) e^{-(a\theta_a -\Psi(\theta_a))t}}{(\E \tau_{+}) \sigma_a (2\pi t)^{1/2}}\\
&\quad \times \int_0^{x_0} \P (D_{\tau_+}>x)e^{-\theta_a x}\big(1+O(\frac{(\log t)^2}{t})\big) dx\\
&\quad +O((n-t+1) \P (T_1>at+x_0))\\
&=\frac{(n-t+1)  e^{-(a\theta_a -\Psi(\theta_a))t}}{(\E \tau_{+}) \sigma_a (2\pi t)^{1/2}}\\
&\quad \times \left[ \int_0^{\infty} \P(D_{\tau_+}> x)e^{-\theta_a x}\big(1+O(\frac{(\log t)^2}{t})\big) dx+O(\frac{1}{t}) \right].
}
By the integration by parts formula, we have
\besn{\label{star9}
&\frac{1}{\E \tau_+} \int_0^{\infty} \P(D_{\tau_+}>x)e^{-\theta_a x} dx\\
&=\frac{1}{\theta_a \E \tau_+}\Big[ 1-\E e^{-\theta_a D_{\tau_+}}  \Big]\\
&=\frac{1}{\theta_a \E \tau_+} \exp\big[-\sum_{k=1}^\infty k^{-1}\E(e^{-\theta_a D_k}, D_k> 0) \big] \\
&=\frac{1}{\theta_a} \exp\big[-\sum_{k=1}^\infty k^{-1}\E(e^{-\theta_a D_k^+}) \big]
}
where we used the first equality in the proof of Corollary 2.7 of \cite{Wo82} and Corollary 2.4 of \cite{Wo82}.
Therefore,
\bes{
\lambda_2&=\frac{ (n-t+1) e^{-(a\theta_a -\Psi(\theta_a))t}}{\theta_a  \sigma_a (2\pi t)^{1/2}}  \\
&\quad\times \exp\big[-\sum_{k=1}^\infty k^{-1}\E (e^{-\theta_a D_k^+}) \big]  \big(1+O(\frac{(\log t)^2}{t})\big).
}

Next we consider Case 2 of Theorem \ref{t1}. The calculation of $\lambda_2$ is similar to Case 1 except that we have, for integers $0\leq k\leq x_0$,
\be{
\P_{\theta_{a}} (T_1= \lceil at\rceil+k)=\frac{1}{\sigma_a (2\pi t)^{1/2}}+O(\frac{(\log t)^2}{t^{3/2}})
}
and
\besn{\label{star8}
&\sum_{k=0}^\infty \P(D_{\tau_+}>\lceil at \rceil -at+k) e^{-\theta_a (\lceil at \rceil -at+k)}\\
&=e^{-\theta_a(\lceil at \rceil -at)}\sum_{k=0}^\infty \P(D_{\tau_+}>k)e^{-\theta_a k}\\
&=\frac{e^{-\theta_a (\lceil at \rceil -at)}}{1-e^{-\theta_a }}\big[ 1-\E e^{-\theta_a D_{\tau_+}} \big].
}
Therefore, for the arithmetic case,
\bes{
&\lambda_2=\frac{ (n-t+1) e^{-(a\theta_a -\Psi(\theta_a))t} e^{-\theta_a (\lceil at \rceil -at)}}
{(1-e^{-\theta_a })\sigma_a (2\pi t)^{1/2}}\\
&\qquad \times \exp\big[-\sum_{k=1}^\infty k^{-1}\E (e^{-\theta_a D_k^+}) \big]  \big(1+O(\frac{(\log t)^2}{t})\big).
}

\end{proof}

\subsection{Proof of Corollary \ref{c1}}
Following the proof of Theorem \ref{t1}, let
\be{
Y_1=\I(X_1=\dots=X_t=a)
}
and for $2\le \alpha\le n-t+1$,
\be{
Y_\alpha=\I(X_{\alpha-1}<a, X_\alpha=\dots=X_{\alpha+t-1}=a).
}
Then with $W=\sum_{\alpha=1}^{n-t+1} Y_\alpha$,
\be{
\E W=p_a^t + (n-t) p_a^t (1-p_a)=\lambda.
}
Instead of \eq{1.1}, we have $\P(M_{n;t}\geq at)=\P(W\geq 1)$, and instead of \eq{1.11}, we have
\be{
|\P(W\geq 1)-(1-e^{-\lambda})|\leq (1\wedge \frac{1}{\lambda}) (n-t+1)(2t+1)p_a^{2t}.
}
This proves the bound \eq{c1-1}.

\subsection{Proof of Theorem \ref{t2}}
In this subsection, let $C$ and $c$ denote positive constants depending on the exponential family $F_\theta$ and $\theta_1$, and may represent different values in different expressions.
The lemmas used in the proof of Theorem \ref{t2} will be stated and proved after the proof.

\begin{proof}[Proof of Theorem \ref{t2}]
Recall $S_i=\sum_{k=1}^i X_k$. Define $\tau_+=\inf\{n\geq 1: S_n>0\}$ and
\ben{\label{1001}
\tau_b:=\inf \{n\geq 1: S_n\geq b\},\quad T_b:= \inf \{n\geq 1: S_n\notin [0, b)\}.
}
If $b$ is bounded, then by choosing $C$ to be large enough in \eq{t2-1},
and observing that
\ben{\label{1002}
\lambda\sim (n-\frac{b}{\mu_1}) e^{-\theta_1 b},
}
we have
$C\lambda b^{1/2} h(b)/(n-b/\mu_1)\geq 1$ and \eq{t2-1} is trivial.
Therefore, in the following we can assume $b$ is larger than any given constant.
Moreover, since we assume $h(b)=O(b^{1/2})$ in the theorem, by choosing $C$ to be large enough
and $c$ to be small enough in \eq{t2-1}, we only need to consider the case where $h(b)/b^{1/2}$ is smaller than any given positive constant. In particular, we can assume
\ben{\label{star6}
\frac{h(b)}{b^{1/2}}\leq \min \Big\{\frac{2}{\mu_1}, 
\frac{2(\theta_1'-\theta_1)\sup_{0\leq \theta\leq \theta_1} \Psi''(\theta)}{\mu_1^2} \Big\}
}  
for some $\theta_1'\in \Theta$ and $\theta_1<\theta_1'<2\theta_1$.

We embed the sequence $\{X_1,\dots, X_n\}$ into an infinite i.i.d. sequence 
\be{
\{\dots,X_{-1},X_0,X_1,\dots\}.
}
For a positive integer $m$, let $\omega_m^+$ be the $m$-shifted sample path of $\omega:=\{X_1,\dots, X_n\}$, so $S_i(\omega_m^+)=S_{m+i}(\omega)-S_m(\omega)$,
$T_b(\omega_m^+)=\inf\{n\geq 1: S_n(\omega_m^+)\notin [0,b)\}$, and $\tau_b(\omega_m^+)$, $\tau_+(\omega_m^+)$ are defined similarly.
Let $t=\lceil \frac{b}{\mu_1}+b^{1/2}h(b) \rceil$ and $m=\lfloor c h^2(b) \rfloor$ such that $m<t$. For $1\leq \alpha\leq n-t$, let
\ben{\label{1101}
Y_\alpha=\I \big( S_\alpha< S_{\alpha-\beta}, \forall\ 1\leq \beta\leq m; \ T_b(\omega_{\alpha}^+)\leq t, \ S_{T_b}(\omega_\alpha^+)\geq b   \big).
}
That is, $Y_\alpha$ is the indicator of the event that the sequence $\{S_i\}$ reaches a local minimum at $\alpha$ and the $\alpha$-shifted sequence $\{S_i(\omega_\alpha^+)\}$ exits the interval $[0,b)$ within time $t$ and the first exiting position is $b$. Let
\ben{\label{903}
W=\sum_{\alpha=1}^{n-t}Y_\alpha.
}
In the following, we first compare $p_{n,b}$ with $\IP (W\geq 1)$. Then, we approximate the distribution of $W$ by the Poisson distribution with mean $\IE (W)$. Finally, we calculate approximately $\IE (W)$. 

First, from the definition of $W$, we have $p_{n,b}\geq \IP (W\geq 1)$ and with $t_1=\lfloor b/\mu_1- b^{1/2} h(b) \rfloor$,
\bes{
&\{\max_{0\leq i<j\leq n} (S_j-S_i)\geq b\}\backslash \{W\geq 1\}\\
&\subset \left( \cup_{k=0}^{n-t-1}\big\{S_{T_b}(\omega_k^+)\geq b, T_b(\omega_k^+)>t  \big\}  \right)\\
&\quad  \cup \left( \cup_{k\in [0,m]\cup (n-t,n-t_1)} \big\{ S_{T_b}(\omega_k^+)\geq b, T_b(\omega_k^+)\leq t  \big\} \right)\\
&\quad \cup \left(\cup_{n-t_1\leq i<j\leq n}\big\{ S_j-S_i\geq b \big\} \right).
}
By symmetry,
\besn{\label{3.1}
& p_{n,b}-\IP (W\geq 1)\\
&\leq (n-t)\IP  (S_{T_b}\geq b, T_b>t)+(m+2b^{1/2} h(b) +2) \IP (S_{T_b}\geq b, T_b\leq t)\\
&\quad +\E \I(\cup_{0\leq i<j\leq t_1}\{S_j-S_i\geq b\} ) .
}
By \eq{l-1-1} and Lemma \ref{l1}, we have
\ben{\label{sstar6}
\P(S_{T_b}\geq b)=\E_{\theta_1}[e^{-\theta_1 S_{T_b}} \I(S_{T_b}\geq b)]\leq e^{-\theta_1 b}
}
and
\besn{\label{sstar7}
&\IP  (S_{T_b}\geq b, T_b>t)=\E_{\theta_1} [e^{-\theta_1 S_{T_b}} \I(S_{T_b}\geq b, T_b>t)]\\
&\leq e^{-\theta_1 b} \P_{\theta_1}(T_b>t)
\leq Ce^{-\theta_1 b-ch^2(b)}.
}
Along with Lemma \ref{l3}, we have
\besn{\label{star10}
&p_{n,b}-\IP (W\geq 1)\\
& \leq C(n-b/\mu_1)e^{-\theta_1 b}
\big\{ e^{-ch^2(b)}+\frac{m+b^{1/2}h(b)}{n-b/\mu_1}+\frac{b/h^2(b)}{n-b/\mu_1} e^{-ch^2(b)}  \big\}.
}

Next, we use Theorem \ref{t3} to obtain a bound on the total variation distance between the distribution of $W$ and $Poi(\lambda_1)$ with 
\ben{\label{904}
\lambda_1:=\E (W)=(n-t)\E Y_\alpha.
}
For each $1\leq \alpha\leq n-t$, let $B_\alpha=\{1\leq \beta\leq n-t: |\beta-\alpha|\leq t+m\}$.
In applying Theorem \ref{t3}, by our definition of $B_\alpha$, $b_3=0$. From $|B_\alpha|\leq 2(t+m)+1$ and \eq{sstar6}, we have
\ben{\label{sstar4}
b_1<[2(t+m)+1]\lambda_1 \E Y_\alpha\leq C(n-t)(t+m)\P^2(S_{T_b}\geq b)\leq C(n-t)(t+m)e^{-2\theta_1 b}.
}
Let
\be{
\tilde{Y}_\alpha=\I(T_b(\omega_\alpha^+)\leq t, S_{T_b}(\omega_\alpha^+)\geq b).
}
We have for $b_2$ in \eq{t3-2},
\bes{
b_2 \leq & \sum_{\alpha=1}^{n-t} \sum_{\alpha\ne \beta\in B_\alpha} \E(Y_\alpha Y_\beta)\\
&\leq 2\sum_{\beta=1}^{n-t} \left[ \sum_{\beta-t-m\leq \alpha<\beta-m} \E Y_\beta \tilde{Y}_\alpha
+\sum_{\beta-m\leq \alpha\leq \beta-1} \E Y_\beta \tilde{Y}_\alpha   \right].
}
For $\beta-t-m\leq \alpha<\beta-m$, because $S_\beta<S_\alpha$ implies $T_b(w_\alpha^+)\leq \beta-\alpha$, we have
\bes{
\E Y_\beta \tilde{Y}_\alpha&=\E Y_\beta \tilde{Y}_\alpha [\I(S_\beta\geq S_\alpha)+\I(S_\beta< S_\alpha)]\\
&\leq \E \I(S_\beta-S_\alpha\geq 0) \tilde{Y}_\beta 
+\E I(S_{T_b}(\omega_\alpha^+)\geq b, T_b(\omega_\alpha^+)\leq \beta-\alpha) \tilde{Y}_\beta.
}
By independence and symmetry, we have
\be{
\sum_{\beta-t-m\leq \alpha<\beta-m}\E Y_\beta \tilde{Y}_\alpha \leq \E \tilde{Y}_1 
\sum_{i=m}^{t+m}[\P(S_i\geq 0)+\P(S_{T_b}\geq b, T_b\leq t+m)].
}
For $\beta-m\leq \alpha\leq \beta-1$, because $Y_\beta=1$ implies $S_\alpha>S_\beta$, which in turn implies $T_b(w_\alpha^+)\leq \beta-\alpha$, we have
\bes{
\sum_{\beta-m\leq \alpha\leq \beta-1} \E Y_\beta \tilde{Y}_\alpha 
&\leq \sum_{\beta-m\leq \alpha\leq \beta-1} \E\tilde{Y}_\beta \I(S_{T_b}(\omega_\alpha^+)\geq b, T_b(\omega_\alpha^+)\leq \beta-\alpha )\\
&\leq \E \tilde{Y}_1 \sum_{i=1}^m \P(S_{T_b}\geq b, T_b\leq i).
}
Therefore, by Lemma \ref{l2} and \eq{sstar6},
\besn{\label{sstar5}
b_2&\leq 2(n-t)\E \tilde{Y}_1 \Big[ \sum_{i=m}^{t+m} \big( \P(S_i\geq 0)+\P(S_{T_b}\geq b, T_b\leq t+m)  \big)\\
&\kern8em +\sum_{i=1}^m \P(S_{T_b}\geq b, T_b\leq i)  \Big]\\
&\leq 2(n-t)e^{-\theta_1 b} \big[ Ce^{-cm} +(t+m) e^{-\theta_1 b}  \big].
}
From \eq{t3-1}, \eq{sstar4} and \eq{sstar5}, we obtain
\ben{\label{3.2}
\big|\P(W\geq 1)-(1-e^{-\lambda_1})\big|\leq C(n-t) e^{-\theta_1 b} \big[ (t+m)e^{-\theta_1 b} +e^{-cm} \big].
}
Since $\lambda_1$ does not have an explicit expression, our final goal is to show that $\lambda_1$ is close to $\lambda$. For this purpose, we first introduce an intermediate quantity $\lambda_2$ defined as
\ben{\label{905}
\lambda_2=(n-t) \IP (\tau_0=\infty)\IP (S_{T_b}\geq b).
}
Recall
\bes{
&\lambda_1=(n-t)\E Y_\alpha\\
&=(n-t)\P(S_{\alpha-\beta}-S_\alpha>0,\ \forall\ 1\leq \beta\leq m) \P(T_b(\omega_\alpha^+)\leq t, S_{T_b}(\omega_\alpha^+)\geq b)\\
&=(n-t)\P(\tau_0>m)\P(T_b\leq t, S_{T_b}\geq b).
}
From the upper and lower bounds of their difference
\be{
\lambda_2-\lambda_1\leq (n-t) \IP (T_b>t, S_{T_b}\geq b),
}
\be{
\lambda_1-\lambda_2\leq (n-t) \IP  (S_{T_b}\geq b) \IP  (m< \tau_0 <\infty),
}
we have
\besn{\label{3.3}
|\lambda_1-\lambda_2|&\leq C(n-t) e^{-\theta_1 b-ch^2(b)}+(n-t) e^{-\theta_1 b}\sum_{i=m}^\infty \P(S_i\geq 0)\\
&\leq C(n-t) e^{-\theta_1 b}[e^{-c h^2(b)}+e^{-cm}]
}
where we used \eq{sstar7}, \eq{sstar6} and Lemma \ref{l1}.

Finally, in Lemma \ref{lemma10}, we 
will show that 
\ben{\label{star11}
\lambda_2=\big[ 1+O(\frac{b^{1/2}h(b)}{n-b/\mu_1}) \big] \lambda +(n-t)e^{-\theta_1 b} o(e^{-cb}).
}
Theorem \ref{t2} is proved by combining \eq{star10}, \eq{3.2}, \eq{3.3} and \eq{star11} and using \eq{1002}.
\end{proof}

The following lemmas have been used in the above proof.
\begin{lemma}\label{l-1}
Let $\{X_1,\dots, X_n\}$ be independent, identically
distributed  random variables with distribution function $F$ that
can be imbedded in an exponential family, as in \eq{1.2}.  Let  $\E X_1=\mu_0<0$. 
Let $S_0=0$ and
$S_i=\sum_{k=1}^i X_k$ for $1\leq i\leq n$.  
Suppose there exist $\theta_1>0$ such that $ \Psi(\theta_1)=0$.
Let $\mathcal{F}_n=\sigma\{X_1,\dots, X_n\}$, and let $T$ be a stopping time with respect to $\{\mathcal{F}_n\}$. Then we have
\ben{\label{l-1-1}
\P(G\cap \{T<\infty\})=\E_{\theta_1} \big[ e^{-\theta_1 S_T} \I(G\cap\{T<\infty\})  \big]
}
for any $G\in \mathcal{F}_T$.
\end{lemma}
\begin{proof}
Equation \eq{l-1-1} follows by a direct application of Wald's likelihood ratio identity (cf. Theorem 1.1 of \cite{Wo82}) to the sequence $\{X_1, X_2, \dots\}$.
\end{proof}
\begin{lemma}\label{l1}
Under the setting of Theorem \ref{t2},
let $t=\lceil \frac{b}{\mu_1}+b^{1/2} h(b) \rceil$. We have
\be{
\P_{\theta_1}(T_b>t)\le  C e^{-c h^2(b)},
}
where $T_b$ is defined in \eq{1001}.
\end{lemma}
\begin{proof}
Let
\be{
r=\frac{\mu_1^2}{2 \sup_{0\leq \theta\leq \theta_1} \Psi''(\theta)} h(b)/b^{1/2}.
}
By \eq{star6}, we have $r<\theta_1$ and 
\ben{\label{901}
\mu_1 r-\sup_{0\leq \theta \leq \theta_1} \Psi''(\theta) r^2/2\geq \mu_1 r/2.
}
Form $\{T_b>t\}\subset \{S_t\leq b\}$ and Markov's inequality, we have
\be{
\P_{\theta_1}(T_b>t)\leq \P_{\theta_1}(S_t\leq b)\leq e^{rb} \E_{\theta_1} e^{-r S_t}.
}
This is further bounded by
\bes{
\P_{\theta_1}(T_b>t)&\leq \exp\big\{ r b-[\Psi(\theta_1)-\Psi(\theta_1-r)]t \big\}\quad \text{(by direct computation)}\\
&\leq \exp\big\{ r b-[\mu_1r -\sup_{0\leq \theta\leq \theta_1} \Psi''(\theta)r^2/2 ]t  \big\}\quad \text{(by Taylor's expansion)}\\
&\leq \exp\big\{\frac{\sup_{0\leq \theta\leq \theta_1} \Psi''(\theta) r^2 b}{2\mu_1}-
[\mu_1 r-\sup_{0\leq \theta\leq \theta_1} \Psi''(\theta) r^2/2]b^{1/2}h(b)   \big\}\\
&\kern16em \text{(from the definition of $t$)}\\
&\leq \exp\big\{\frac{\sup_{0\leq \theta\leq \theta_1} \Psi''(\theta)r^2 b}{2\mu_1} 
- \frac{\mu_1r}{2} b^{1/2} h(b)    \big\}\quad \text{(from \eq{901})}\\
&\leq \exp\big\{-\frac{\mu_1^3}{8 \sup_{0\leq \theta\leq \theta_1} \Psi''(\theta)}h^2(b)   \big\}\quad \text{(from the definition of $r$)}.
}
This proves Lemma \ref{l1}.
\end{proof}
\begin{lemma}\label{l2}
Under the setting of Theorem \ref{t1}, for positive integers $m$, we have
\be{
\sum_{i=m}^\infty \P(S_i\geq 0)\leq Ce^{-cm}.
}
\end{lemma}
\begin{proof}
Lemma \ref{l2} follows from
\be{
\P(S_i\geq 0)\leq \E e^{\theta^* S_i}=e^{\Psi(\theta^*) i},
}
where $0<\theta^*<\theta_1$ and $\Psi(\theta^*)<0$.
\end{proof}
\begin{lemma}\label{l3}
Under the setting of Theorem \ref{t2}, let $t_1=\lfloor \frac{b}{\mu_1} - b^{1/2}h(b) \rfloor$. We have
\be{
\E \I(\cup_{0\leq i<j\leq t_1}\{S_j-S_i\geq b\} ) \leq Ce^{-\theta_1 b} \frac{b}{h^2(b)} e^{-c h^2(b)}.
}
\end{lemma}
\begin{proof}
We only need to consider the case when $t_1>0$.
Let
\be{
r=\frac{\mu_1^2}{2\sup_{\theta_1 \leq\theta\leq\theta_1'}\Psi''(\theta)} h(b)/b^{1/2},
}
where $\theta_1'$ is defined just below \eq{star6}.
By \eq{star6}, $\theta_1+r\leq \theta_1' \in \Theta$.
By \eq{1.2},
we have
\bes{
\P_{\theta_1}(S_j\geq b)=&\E_{\theta_1+r} \big\{ e^{-r S_j +j[\Psi(\theta_1+r)-\Psi(\theta_1)]} I(S_j\geq b)  \big\} \\
\leq &\exp\big\{ j[\Psi(\theta_1+r)-\Psi(\theta_1)]-r b  \big\},
}
Thus,
\besn{\label{902}
\sum_{j=1}^i \P_{\theta_1}(S_j\geq b)&\leq \frac{1}{1-e^{\Psi(\theta_1)-\Psi(\theta_1+r)}} \exp \big\{ i[ \Psi(\theta_1+r)-\Psi(\theta_1)]-r b  \big\}.
}
By the union bound, \eq{l-1-1} and \eq{902}, we have
\bes{
&\E \I(\cup_{0\leq i<j\leq t_1}\{S_j-S_i\geq b\} ) \leq e^{-\theta_1 b}\sum_{i=1}^{ t_1} \sum_{j=1}^i \P_{\theta_1}(S_j\geq b)\\
&\leq \frac{ e^{-\theta_1 b}}{1-e^{\Psi(\theta_1)-\Psi(\theta_1+r)}}\sum_{i=1}^{ t_1} \exp \big\{ i [ \Psi(\theta_1+r)-\Psi(\theta_1)]-r b  \big\}\\
&\leq e^{-\theta_1 b} \Big( \frac{1}{1-e^{\Psi(\theta_1)-\Psi(\theta_1+r)}} \Big)^2\exp \big\{ t_1 [ \Psi(\theta_1+r)-\Psi(\theta_1)]-r b  \big\}.
}
By the definition of $r$ and $t_1$, the inequality $(1/(1-e^{-x}))^2\leq Cx^{-2}$ for bounded $x$, and Taylor's expansion, the above bound can be further bounded as
\bes{
&\E \I(\cup_{0\leq i<j\leq t_1}\{S_j-S_i\geq b\} ) \\
&\leq C e^{-\theta_1 b}\frac{b}{h^2(b)} \exp \big\{ t_1 [\Psi(\theta_1+r)-\Psi(\theta_1)]-r b  \big\}\\
&\leq C e^{-\theta_1 b} \frac{b}{h^2(b)} \exp\big\{ (\frac{b}{\mu_1}-b^{1/2}h(b))(r \mu_1+\frac{\sup_{\theta_1 \leq\theta\leq\theta_1'}\Psi''(\theta)}{2}r^2)-r b   \big\}\\
&\leq C e^{-\theta_1 b} \frac{b}{h^2(b)} \exp \big\{ -b^{1/2}h(b)r \mu_1+
\frac{\sup_{\theta_1 \leq\theta\leq\theta_1'}\Psi''(\theta) b}{2\mu_1} r^2   \big\}\\
&\leq C e^{-\theta_1 b} \frac{b}{h^2(b)}e^{-ch^2(b)}.
}
This proves Lemma \ref{l3}.
\end{proof}
The next lemma will be used in proving Lemma \ref{lemma10}.
\begin{lemma}\label{l4}
Under the setting of Theorem \ref{t2}, if $\int_{-\infty}^\infty |\varphi_{\theta_1} (t)|dt<\infty$ where $\varphi_{\theta_1}(t)=\E_{\theta_1} e^{it X_1}$,
then $S_{\tau_+}$ under $F_{\theta_1}$ has bounded density and is strongly nonarithmetic in the sense that
\be{
\liminf_{|\lambda|\rightarrow \infty} |1-\phi_{\theta_1}(\lambda)|>0,\ \text{where}\  \phi_{\theta_1}(\lambda)=\E_{\theta_1} e^{i\lambda S_{\tau_+}},
}
where $\tau_+$ is defined just above \eq{1001}. 
\end{lemma}
\begin{proof}
The condition $\int_{-\infty}^\infty |\varphi_{\theta_1} (t)|dt<\infty$ implies that under $F_{\theta_1}$, $X_1$ is strongly nonarithmetic.
By (8.42) of \cite{Si85} with $s=1$, the distribution of $S_{\tau_+}$ under $F_{\theta_1}$ is also strongly nonarithmetic.
The condition $\int_{-\infty}^\infty |\varphi_{\theta_1} (t)|dt<\infty$ also implies that the density of $X_1$ under $F_{\theta_1}$ is bounded by a constant $M$. Therefore,
\bes{
\P_{\theta_1}(S_{\tau_+}\in [x, x+dx])& \leq \E_{\theta_1} \sum_{n=0}^\infty
\I(S_1,\dots, S_n\leq 0, S_{n+1}\in [x, x+dx])\\
&\leq\E_{\theta_1} \sum_{n=0}^{\infty} \I(S_n\leq 0, S_{n+1}\in [x, x+dx])\\
&=\sum_{n=0}^\infty \int_{-\infty}^0 \P_{\theta_1}(S_n=dt) \P_{\theta_1}(X_1\in [x-t, x+dx-t])\\
&\leq M dx \sum_{n=0}^\infty \P_{\theta_1}(S_n\leq 0)\leq Cdx,
}
where in the last inequality we used
\ben{\label{sstar8}
\P_{\theta_1}(S_n\leq 0)\leq e^{\Psi(\theta_1-\theta^*) n}
}
for $0<\theta^*< \theta_1$ such that $\Psi(\theta_1-\theta^*) <0$.
This proves that $S_{\tau_+}$ under $F_{\theta_1}$ has bounded density.
\end{proof}

\begin{lemma}\label{lemma10}
Under the setting of Theorem \ref{t2}, \eq{star11} holds.
\end{lemma}
\begin{proof}
From \eq{905} and \eq{l-1-1}, we have
\ben{\label{star7}
\lambda_2 =(n-t)e^{-\theta_1 b} \IP (\tau_0=\infty) \E_{\theta_1}\big( e^{-\theta_1 (S_{T_b}-b)}, S_{T_b}\geq b \big) .
}
Since
\bes{
\E_{\theta_1} \big(e^{-\theta_1 (S_{\tau_b}-b)}  \big)&=\E_{\theta_1} \big(e^{-\theta_1 (S_{\tau_b}-b)}, S_{T_b}\geq b  \big)\\
&\quad +\E_{\theta_1} \big(e^{-\theta_1 (S_{\tau_b}-b)}, S_{T_b}< 0  \big),
}
we have
\besn{\label{2.3}
&\E_{\theta_1} \big(e^{-\theta_1 (S_{T_b}-b)}, S_{T_b}\geq b  \big)\\
&=\E_{\theta_1} \big(e^{-\theta_1 (S_{\tau_b}-b)} \big) - \E_{\theta_1} \big(e^{-\theta_1 (S_{\tau_b}-b)}, S_{T_b}< 0  \big)\\
&=\E_{\theta_1} \big(e^{-\theta_1 (S_{\tau_b}-b)} \big) - \E_{\theta_1} \Big \{ \E_{\theta_1} \big( e^{-\theta_1 (S_{\tau_b}-b)}\big|  S_{T_b}   \big), S_{T_b}< 0   \Big\}.
}
We first consider Case 1. 
Let $\tau_+^{(0)}=0$, and let $\tau_+^{(k)}$ be defined recursively as
$\tau_+^{(k+1)}=\inf \{n>\tau_+^{(k)}: S_n>S_{\tau_+^{(k)}}\}$.
Define $U(x)=\sum_{k=0}^\infty \P_{\theta_1}(S_{\tau_+^{(k)}}\leq x)$.
Observe that $\{S_{\tau_+^{(k+1)}}-S_{\tau_+^{(k)}}, k=0,1,\dots\}$ are i.i.d. with the same distribution as $S_{\tau_+}$.
By Lemma \ref{l4} and (2) of \cite{St65}, we have
\ben{\label{star13}
U(x)= \frac{x}{\E_{\theta_1} S_{\tau_+}}+\frac{\E_{\theta_1} (S_{\tau_+}^2)}{2(\E_{\theta_1} S_{\tau_+})^2}+o(e^{-cx}),\ \text{as}\ x\rightarrow \infty.
}
Following the proof of Corollary 8.33 of \cite{Si85}, we have for $x\geq 0$,
\besn{\label{star14}
&\P_{\theta_1}(S_{\tau_b}-b>x)\\
&=\sum_{n=0}^\infty \P_{\theta_1} (S_{\tau_+^{(n)}}<b ,S_{\tau_+^{(n+1)}}>b+x )\\
&=(\int_{(0,b/2]} +\int_{(b/2, b)}) U(dt) \P_{\theta_1} (  S_{\tau_+}>b+x-t )\\
&=O(\P_{\theta_1} ( S_{\tau_+}>b/2) U(b/2)) + \int_{(b/2, b)} U(dt) \P_{\theta_1} (  S_{\tau_+}>b+x-t ).
}
For $x>0$, 
\bes{
\P_{\theta_1}(S_{\tau_+}>x)&=\E_{\theta_1} [\sum_{i=0}^\infty \I(S_0,\dots, S_i\leq 0, X_{i+1}>x-S_i)]\\
&\leq \sum_{i=0}^\infty \P_{\theta_1} (S_i\leq 0, X_{i+1}>x)\\
&= \sum_{i=0}^\infty \P_{\theta_1}(S_i\leq 0)\P_{\theta_1} (X_{i+1}>x)\leq C \P_{\theta_1}(X_{1}>x)
}
where we used \eq{sstar8}. Therefore,
the right tail probability of $S_{\tau_+}$ under $F_{\theta_1}$ decays exponentially.
From this fact and \eq{star13}, the first term on the right-hand side of \eq{star14} is bounded by $o(e^{-cb})$.
Let $j=\lceil e^{cb} \rceil$ with small enough $c$, and let $\Delta=\frac{b}{2j}$.
Let 
\be{
A=\sum_{k=1}^{j} [U(b-(k-1)\Delta)-U(b-k\Delta)] \P_{\theta_1}(S_{\tau_+}>x+k \Delta).
}
We have
\be{
\int_{(b/2, b)} U(dt) \P_{\theta_1} (  S_{\tau_+}>b+x-t )\geq A
}
and by \eq{star13} and the fact that $S_{\tau_+}$ under $F_{\theta_1}$ has bounded density (cf. Lemma \ref{l4}),
\bes{
&\int_{(b/2, b)} U(dt) \P_{\theta_1} (  S_{\tau_+}>b+x-t )-A\\
&\leq \sum_{k=1}^j [U(b-(k-1)\Delta)-U(b-k\Delta) ]
\P_{\theta_1}(S_{\tau_+}\in [x+(k-1)\Delta, x+k\Delta] )\\
&=o(e^{-cb}).
}
From \eq{star13},
\ben{\label{1201}
A=\sum_{k=1}^j \frac{\Delta}{\E_{\theta_1} S_{\tau_+}}\P_{\theta_1} (S_{\tau_+}>x+k\Delta)+O(j e^{-cb})
}
with the same constant $c$ as in \eq{star13}. 
By choosing $c$ in the definition of $j$ to be small enough, the second term on the right-hand side of \eq{1201} is of smaller order of $e^{-cb}$.
Using the fact that $S_{\tau_+}$ under $F_{\theta_1}$ has bounded density and an exponential tail, we have
\be{
\sum_{k=1}^j \frac{\Delta}{\E_{\theta_1} S_{\tau_+}}\P_{\theta_1} (S_{\tau_+}>x+k\Delta)
=\frac{1}{\E_{\theta_1} S_{\tau_+}} \int_{x}^\infty \P_{\theta_1} (S_{\tau_+}>y) dy +o(e^{-cb}).
}
Therefore,
\bes{
A&=\sum_{k=1}^{j} [\frac{\Delta}{\E_{\theta_1} S_{\tau_+}} +o(e^{-cb})] \P_{\theta_1}(S_{\tau_+}>x+k \Delta)\\
&=\frac{1}{\E_{\theta_1} S_{\tau_+}} \int_{x}^\infty \P_{\theta_1} (S_{\tau_+}>y) dy +o(e^{-cb}).
}
By \eq{star14} and the above argument, we have
\ben{\label{star15}
\P_{\theta_1}(S_{\tau_b}-b>x)=
\frac{1}{\E_{\theta_1} S_{\tau_+}} \int_{x}^\infty \P_{\theta_1} (S_{\tau_+}>y) dy +o(e^{-cb}).
}
Using the integration by parts formula, \eq{star15} and \eq{star9}, we obtain 
\besn{\label{2.4}
&\E_{\theta_1} \big(e^{-\theta_1 (S_{\tau_b}-b)} \big)\\
&=1-\theta_1 \int_0^{\infty} \P_{\theta_1}(S_{\tau_b}-b>x) e^{-\theta_1 x}dx\\
&=1-\frac{\theta_1}{\E_{\theta_1} S_{\tau_+}}\int_0^\infty \int_x^\infty \P_{\theta_1} (S_{\tau_+}>y) e^{-\theta_1 x}dydx
+o(e^{-cb})\\
&=1+\frac{1}{\E_{\theta_1} S_{\tau_+}} \int_0^\infty (e^{-\theta_1 y}-1) \P_{\theta_1}(S_{\tau_+}>y)dy +o(e^{-cb})\\
&=\frac{1}{\mu_1 \E_{\theta_1}\tau_+}\int_0^\infty e^{-\theta_1 y}\P_{\theta_1}(S_{\tau_+}>y)dy+o(e^{-cb})\\
&=\frac{1}{\theta_1 \mu_1} \exp\big( -\sum_{k=1}^\infty \frac{1}{k}\E_{\theta_1} e^{-\theta_1 S_k^+} \big)+o(e^{-cb}).
}
Choosing $\theta^*$ such that $0<\theta^*<\theta_1$ and $\Psi(\theta_1-\theta^*)<0$, we have
\besn{\label{1301}
&0\leq \P_{\theta_1}(S_{T_b}\geq b)-\P_{\theta_1}(\tau_-=\infty)\\
&\leq \sum_{i=1}^\infty \P_{\theta_1}(S_i<-b)\leq e^{-\theta^*b} \sum_{i=1}^\infty e^{\Psi(\theta_1-\theta^*)i}=o(e^{-cb}).
}
From \eq{2.3}, \eq{2.4} and \eq{1301}, we have, with $\tau_{-}:=\inf \{n: S_n< 0\}$,
\bes{
&\E_{\theta_1} \big(e^{-\theta_1 (S_{T_b}-b)}, S_{T_b}\geq b  \big)\\
&=\frac{1}{\theta_1 \mu_1} \exp \big( -\sum_{n=1}^{\infty} \frac{1}{n} \E_{\theta_1} e^{-\theta_1  S_n^+ } \big) \P_{\theta_1}(S_{T_b}\geq b) + o(e^{-c b})\\
&= \frac{1}{\theta_1 \mu_1} \exp \big( -\sum_{n=1}^{\infty} \frac{1}{n} \E_{\theta_1} e^{-\theta_1  S_n^+ } \big) \P_{\theta_1}(\tau_{-}=\infty) + o(e^{-c b})
}
as $b\rightarrow \infty$.

From Lemma \ref{l-1} and Corollary 2.4 of \cite{Wo82}, we have
\besn{\label{1302}
&\P(\tau_0=\infty) \P_{\theta_1}(\tau_-=\infty)
=\exp\big\{-\sum_{k=1}^\infty \frac{1}{k} [\P(S_k\geq 0)+\P_{\theta_1}(S_k<0)] \big\}\\
&=\exp\big\{ -\sum_{k=1}^\infty \frac{1}{k} [\P_{\theta_1}(e^{-\theta_1 S_k},S_k\geq 0)+\P_{\theta_1}(S_k<0)]  \big\}\\
&=\exp\big\{ -\sum_{k=1}^\infty \frac{1}{k} \E_{\theta_1} e^{-\theta_1 S_k^+}  \big\}.
}
From \eq{star7} and \eq{1302}, we have
\bes{
\lambda_2&=(n-t) e^{-\theta_1 b} \big\{ \frac{1}{\theta_1 \mu_1}  \exp \big( -2\sum_{n=1}^{\infty} \frac{1}{n} \E_{\theta_1} e^{-\theta_1  S_n^+ } \big)  +o(e^{-c b})  \big\}\\
&=\big[ 1+O(\frac{b^{1/2}h(b)}{n-b/\mu_1}) \big] \lambda +(n-t)e^{-\theta_1 b} o(e^{-cb}).
}
Next, we consider Case 2. 
By a similar and simpler argument as for \eq{star15}, we obtain, for integers $k\geq 0$,
\bes{
\P_{\theta_1}(S_{\tau_b}-b=k)
&=\sum_{n=0}^\infty \P_{\theta_1}(S_{\tau_+^{(n)}}<b, S_{\tau_+^{(n+1)}}=b+k)\\
&=\sum_{m=1}^{b-1} [\sum_{n=0}^\infty \P_{\theta_1}(S_{\tau_+^{(n)}}=m)] \P_{\theta_1} (S_{\tau_+}=b+k-m)\\
&=O\Big(\sum_{m=1}^{\lfloor b/2 \rfloor}\sum_{n=0}^\infty 
 \P_{\theta_1}(S_{\tau_+^{(n)}}=m) \P_{\theta_1} (S_{\tau_+}\geq \lfloor b/2 \rfloor) \Big)\\
&\quad +\sum_{m=\lfloor b/2 \rfloor+1}^{b-1} \P_{\theta_1} (S_{\tau_+}=b+k-m)
 \big(\frac{1}{\E_{\theta_1}(S_{\tau_+})}+o(e^{-cb})   \big)\\
&=\frac{1}{\E_{\theta_1} S_{\tau_+}}\P_{\theta_1}(S_{\tau_+}>k)+o(e^{-cb}).
}
By the above equality and \eq{star8}, we have
\bes{
\E_{\theta_1}(e^{-\theta_1 (S_{\tau_b}-b)})
&=\sum_{k=0}^\infty e^{-\theta_1 k}
\frac{1}{\E_{\theta_1} S_{\tau_+}} \P_{\theta_1}(S_{\tau_+}>k)+o(e^{-cb})\\
&=\frac{1}{\mu_1 \E_{\theta_1} \tau_+ (1-e^{-\theta_1})}[1-\E_{\theta_1} e^{-\theta_1 S_{\tau_+}}]+o(e^{-cb})\\
&=\frac{1}{\mu_1(1-e^{-\theta_1 })}
\exp \big( -\sum_{n=1}^{\infty} \frac{1}{n} \E_{\theta_1} e^{-\theta_1  S_n^+ } \big)+o(e^{-cb}).
}
Similar calculation as for Case 1 yields
\bes{
\lambda_2=&(n-t) e^{-\theta_1 b} \big\{ \frac{1}{(1-e^{-\theta_1 })\mu_1} \exp \big( -2\sum_{n=1}^{\infty} \frac{1}{n} \E_{\theta_1} e^{-\theta_1  S_n^+ } \big)  +o(e^{-c b})  \big\}\\
=&\big[ 1+O(\frac{b^{1/2}h(b)}{n-b/\mu_1}) \big] \lambda +(n-t)e^{-\theta_1 b} o(e^{-cb}).
}
The lemma is now proved.
\end{proof}

\section{Discussion}

The arguments we used to prove Theorem \ref{t1} and Theorem \ref{t2} may be 
useful in proving rates of convergence for tail probabilities of other test 
statistics for detecting local signals in sequences of independent 
random variables. Two for which some new techniques 
will be needed are
the generalized likelihood ratio statistic
and the Levin and Kline statistic (\cite{LeKl85}).

For example, let $\{X_1,\dots, X_n\}$ be independent random variables from 
the exponential family \eq{1.2}. Consider the testing problem at the beginning 
of the introduction. If the mean of $X_1$ is known and without
loss of generality equal to 0,
the generalized likelihood ratio statistic is
$\max_{1 \leq i < j \leq n }\sup_\theta [\theta (S_j - S_i) - (j-i)\Psi(\theta)]$,
where we have assumed without loss of generality that
$\Psi(0) = 0 = \Psi'(0).$
\cite{SiVe95} derived an asymptotic approximation for
the tail probability of this statistic in the normal case,
while \cite{SiYa00} obtained similar results for a general
exponential family.
The bounds in (24) and (25) of \cite{SiYa00} suggest that the contribution to the maximum from $i$ and $j$ such that $j-i$ is large can be neglected.
However, how to define a local indicator function as in \eq{1101} which avoids `clumping' of $1$'s and to evaluate approximately its expectation remains an open question.

If the mean of $X_1$ is unknown, the statistic is more complicated;
and its tail probability should be evaluated conditionally, given the
value of $S_n$, which is a sufficient statistic for the unknown value
of $\theta$ under the null hypothesis of no change-point.  
The random variables $\{X_1,\dots, X_n\}$ given $S_n$ are globally dependent. 
In applying Theorem \ref{t3} to a sum of Bernoulli random variables $W$
defined similarly as in \eq{903}, the error term $b_3$ is no longer zero,
although we believe it is small. Moreover, it is more challenging to derive a bound as in \eq{3.3} conditionally.

Besides the distribution of the scan statistic $M_{n;t}$ in \eq{0.1},
one may also be interested in the distribution of 
\be{
N^+(b):=\sum_{i=1}^{n-t+1} I(X_i+\dots X_{i+t-1}\geq b).
}
In fact, the distribution of $M_{n;t}$ can be deduced from that of $N^+(b)$ by the relation
\be{
\{M_{n;t}\geq b\}=\{N^+(b)\geq 1\}.
}
Theorem 2 of \cite{DeKa92} gives a Poisson approximation result for $N^+(b)$. 
However, as discussed in Remark \ref{r4}, their approximation may not be adequate because of the `clumping' phenomenon. A more suitable choice of the limiting distribution is a compound Poisson distribution.
Stein's method has been used to prove error bounds for the compound Poisson approximation for sums of Bernoulli random variables. See, for example, \cite{BaChLo92}. By combining Stein's method with the analysis in this paper, one may be able to prove a compound Poisson approximation result for $N^+(b)$.


\

{\bf Acknowledgement}. 
We would like to thank the Associate Editor and the referees for
their helpful comments and suggestions which have significantly improved the
presentation of this paper.

\setlength{\bibsep}{0.5ex}
\def\bibfont{\small}





\end{document}